\documentclass[reqno,12pt,a4paper,sortcompress]{elsarticle}  \nopreprintlinetrue 
\usepackage {enumitem}

 \usepackage{caption}
 \usepackage{subcaption}

\usepackage{amsthm,amssymb}
\usepackage{mathrsfs,amsfonts,dsfont,functan,extarrows,mathtools,cases}
\usepackage[colorlinks]{hyperref}
\usepackage{graphicx}
\usepackage{float} 
\theoremstyle{definition}

\newtheorem{theorem}{Theorem}[section]
\newtheorem{proposition}[theorem]{Proposition}

\def\ba {\begin{eqnarray}}
\def\ea {\end{eqnarray}}

\theoremstyle{definition}
\newtheorem{remark}[theorem]{Remark}
\newtheorem{lemma}[theorem]{Lemma}

\numberwithin{equation}{section}
\allowdisplaybreaks

\def\d{\mathop{}\!\mathrm{d}}

\begin{document}

\captionsetup[figure]{labelfont={bf},labelformat={default},labelsep=period,name={Figure}}

\captionsetup[figure]{labelfont={bf},name={Fig.},labelsep=period}
\begin{frontmatter}
\title {Dynamical analysis in a nonlocal delayed reaction-diffusion \\
tumor model with therapy}

\author[U]{Dandan Hu}
\ead{dandanh@mun.ca}
\author[U]{Yuan Yuan\corref{cspd}}
\cortext[cspd]{Corresponding author}
\ead{yyuan@mun.ca}
\address[U]{
Department of Mathematics and Statistics, Memorial University of Newfoundland, St. John’s, NL A1C 5S7, Canada }

\begin{abstract}

In this work, we investigate the dynamical properties of a reaction-diffusion system arising from
tumor-therapy modelling that features both nonlinear interactions and nonlocal delay.
By applying the Lyapunov-Schmidt
reduction, we establish the existence of a nontrivial steady-state solution bifurcating
from the trivial solution. In particular, we derive an approximate expression for a
spatially nonhomogeneous steady-state solution. 
Then, we provide a detailed spectral characterization of the linearized operator and explicit stability criteria and identify the delay-dependent Hopf bifurcation regimes.
To illustrate the theoretical results, we include a concrete example that verifies
the claims in our theorems and numerically demonstrates how changes in treatment parameters alter
stability and bifurcation behaviour.

\end{abstract}

\begin{keyword}
Reaction-diffusion; Nonlocal delay; Lyapunov-Schmidt reduction;  
Stability;  Hopf bifurcation
 \sep Tumor therapy  

\vspace{0.1cm}
\MSC[2020] Primary 34K18, 35B32, 35B35; Secondary 35K57, 35Q92
\end{keyword}

\end{frontmatter}


\section{Introduction}\label{sec:intro}

Spatial heterogeneity   models based on partial differential equations (PDEs) are widely used to describe the proliferation and invasive spread of tumor cell populations \cite{kr2008,kr2000,rr2010}. Incorporating spatial heterogeneity, nonlocal dispersal, and time delays enables these models to capture essential features of the tumor microenvironment that are absent in spatially homogeneous frameworks. In particular, reaction-diffusion equations with delays and nonlocal terms have become standard tools for analyzing bifurcation, pattern formation, and oscillatory dynamics in biological systems. They have proven effective in modeling population dynamics under realistic conditions \cite{yy2011,css2016,gsj2018,syl2019,sjp2019,wh2020,lyq2023}.

At the cellular level, tissue maintenance relies on regulated cycles of cell division and adhesion. Genetic mutations or disruptions of these processes can generate tumor cells that divide uncontrollably \cite{No2022}. Moreover, reduced expression of adhesion molecules such as E-cadherin promotes detachment from the primary tumor, entry into lymphatic or vascular systems, and subsequent colonization of distant sites \cite{sa2009,lm2021}. These biological mechanisms drive invasion and metastasis and should be represented in models that aim to inform therapeutic strategies. Clinical treatments include surgery, radiotherapy, chemotherapy, and immunotherapy. To quantify the impact of therapy on tumor growth dynamics while maintaining population conservation principles \cite{siam2021,murry2007}, a general framework expresses tumor cell evolution as:
$$\text{Rate of tumor cell change}=\text{Diffusion}+\text{Net proliferation}  - \text{Therapy-induced death}.$$ 
Despite extensive modeling efforts \cite{sk2003,ymk2021,skr2002,bf2012,bas2014,mc2019,kcy2024}, significant gaps remain between theory and clinical application. 
In particular, several tumor growth models neglect maturation delays associated with the cell cycle, despite their critical role in tumor development. Moreover, nonlocal dispersal—an essential mechanism driving tumor invasion and recurrence—is rarely incorporated or systematically analyzed. These limitations highlight the need for more comprehensive modeling frameworks that more accurately capture the underlying biological and therapeutic complexities.

Motivated by these considerations, we extend the framework of 
\cite{css2016, JMB2009} and propose a nonlocal delayed reaction-diffusion tumor model with spatially heterogeneous therapy. Denoting tumor cell density at location $x$ and time $t$ by $u(x, t)$, 
we consider the system
\begin{equation}\label{1.1}
\overbrace{\frac{\partial u(x,t)}{\partial t}}^{\begin{array}{c}
\text { rate of change } \\
\text { of tumor cell } \\
\text { concentration }
\end{array}}=\overbrace{d \Delta u(x,t)}^{\begin{array}{c}
\text { net diffusion } \\
\text { of tumor cells }
\end{array}}+\overbrace{   F\left(u(x, t), \int_{\Omega} \mathcal{S}(x, y)  u(y, t-\tau)  \d y\right) u(x, t)}^{\begin{array}{c}
\text { net proliferation of tumor cells }
\end{array}}-\overbrace{ R(x,t)u(x,t),}^{\begin{array}{c}
\text { death due } \\
\text { to therapy }
\end{array}}
\end{equation}
where  $d>0$ is the diffusion rate of tumor cells, $ \Delta$ denotes the Laplacian operator on $\mathbb{R}^n  (1 \leq n \leq 3)$ modelling spatial dispersal, $\tau \geq 0$ represents the maturation delay: the time required for a newborn tumor cell to become mitotically active, the kernel $\mathcal{S}(x, y)$ encodes the probability that a tumor cell born at $y$ survives the delay and relocates to $x$.
The function $F( \cdot , \cdot )$ governs proliferation 
and admits various nonlinear forms from the literature \cite{wh2020, 
Britton1990}. For instance,
$$
F( \cdot , \cdot )=1+a u(x, t)-(1+a) \int_{\Omega} \mathcal{S}(x, y) u(y, t-\tau) \d y,
$$
or 
$$
F( \cdot , \cdot )=1+a u(x, t)-b u(x, t)^2-(1+a-b) \int_{\Omega} \mathcal{S}(x, y) u(y, t-\tau) \d y,
$$
with $b>0$ and $1+a-b>0$.
 The term 
$R(x,t)$ models therapy-induced cell death, which may be spatially heterogeneous, constant, periodic, or absent. Appropriate initial and boundary conditions are imposed on a smooth, bounded domain 
$\Omega\in \mathbb{R}^n (1 \leq n \leq 3)$ with smooth boundary $\partial \Omega$ and initial value 
$$u(x, \theta )=u_0(x, \theta ),  \, x \in \bar{\Omega}, \, \theta \in[-\tau, 0].$$

To analyze the model \eqref{1.1} mathematically, we give the following assumptions:
\begin{enumerate}
\item[\textbf{(S$_1$)}]   $F(u,v) \in C^m,
m \geq 3.$
\item[\textbf{(S$_2$)}]  The kernel function $\mathcal{S}(x, y)$ is a continuous and nonnegative function on $\bar{\Omega} \times \bar{\Omega}$, and the linear Fredholm integral operator
$$
L(\psi(x)):=\int_{\Omega} \mathcal{S}(x, y) \psi(y) \d y
$$
is strictly positive on $C_{+}(\bar{\Omega})$ in the sense that $L\left(C_{+}(\bar{\Omega}) \backslash\{0\}\right) \subset C_{+}(\bar{\Omega}) \backslash\{0\}$, see \cite{zxq2009}.
\end{enumerate}
We adopt the following notations:
$\mathbb{R}^{+}$ and $\mathbb{C}$ denote the positive real numbers and complex numbers; $\mathbb{N}=\{1,2, \cdots\}$; $ \mathbb{N}_0=$ $\{0\} \cup \mathbb{N}$; 
Dom$(L)$, Ker$(L)$ and Range$(L)$ denote the domain, kernel and range of a linear operator $L$; the complexification of a space $Z$ is $Z_{\mathbb{C}}:=Z \oplus \mathrm{i} Z=\left\{x_1+\mathrm{i} x_2 \mid x_1, x_2 \in Z\right\}$; 
the inner product on the Hilbert space $X_{\mathbb{C}}$ is 
$\langle u, v\rangle=\int_{\Omega} \bar{u}(x) v(x) d x$; $\mathcal{C}^k\left([-\tau, 0], X_{\mathbb{C}}\right)$ is the Banach space of $k$-times continuously differentiable mappings from $[-\tau, 0]$ to $\mathbb{X}_{\mathbb{C}}$;
$H^m(\Omega)(m \geq 0)$ is the Sobolev space of $L^2$-functions with derivatives up to order $m$ in $L^2(\Omega)$;
$H_0^1(\Omega)=\left\{u \in H^1(\Omega) \mid \mathcal{B} u(x)=0 \text { for all } x \in \partial \Omega\right\}$; and 
$\mathbb{X}= H^2(\Omega)  \cap H_0^1(\Omega) $ and $ \mathbb{Y}= L^2(\Omega)$. 

 This work aims to explore the dynamical behavior of the nonlinear system \eqref{1.1}, with a focus on the stability and bifurcation of the nonhomogeneous steady states. 
The remainder of the paper is organized as follows. Section 2 proves the existence of the bifurcating positive steady-state solution and presents an approximate form of such a solution.
Section 3 is devoted to the local stability analysis of the steady-state solutions, and a Hopf bifurcation study is provided in Section 4 around the positive steady states. Section 5 contains an example and numerical simulations that corroborate the theory and discuss implications for treatment.

\section{Existence of Steady-state Solutions}

We first examine the existence of biologically feasible steady states in system \eqref{1.1}. To simplify the analysis, we specify the therapy term as
$$R(x,t)= \begin{cases} 
0, & \text{ without  therapy, } \\  
 \beta q(u) + r(x),  &  \text{ with therapy. }
\end{cases}$$
where $\beta \in \mathbb{R}^+$ represents the treatment rate,
$q(u) \in C^{2}$ satisfies $q(0)=1$ and $q^\prime(u)<0$,
 and $r(x)$ is a smooth function describing the 
spatial variation in therapeutic efficacy. Thus, treatment efficacy decreases with tumor cell density and depends on location.

For $R(x,t)=0$ (i.e. no therapy), the analysis is similar to that in \cite{css2016}.  In this work, we focus on the dynamics with therapy. 
If there exists a steady-state solution
$u(x)$ in \eqref{1.1},
 it must satisfy
 \begin{equation}\label{1.2}
     \left\{\begin{array}{lc}
d \Delta u(x)+   F\left( u(x), \int_{\Omega} \mathcal{S}(\cdot, y) u(y) \d y\right) u(x) - ( \beta q(u(x)) + r(x)) u(x) =0, & x \in \Omega, \\
\mathcal{B}u(x)=0, & x \in \partial \Omega.
\end{array}\right.
 \end{equation}
where 
$\mathcal{B}$ denotes either the homogeneous Neumann condition (isolated habitat) or Dirichlet condition (hostile environment). 

Define the nonlinear operator $\Gamma: \mathbb{X} \times \mathbb{R}^{+} \rightarrow \mathbb{Y}$  by
$$
\Gamma(u,\beta)= d\Delta u +   F\left(u, \int_{\Omega} \mathcal{S}(\cdot, y) u(y) \d y\right) u -  ( \beta q(u) + r(x)) u,
$$
for $u \in \mathbb{X}$ and $\beta \in \mathbb{R}^{+} $. 
Then we try to solve $\Gamma(u,\beta)=0.$
First of all, it is easy to see that, 
there always exists a trivial solution $u(x)=0$ in $\Gamma(u,\beta)=0$ ,  i.e.,  $\Gamma(0,\beta)=0$ for any $\beta \in \mathbb{R}^{+}$. 
To show the existence of nontrivial solutions 
near the trivial solution $(0, \beta)$, we take $\beta$ as a bifurcation parameter, 
then
by applying the implicit function theorem, that is, to express $u(x)$ as an implicit function of $\beta$ in $\Gamma(u,\beta)=0$. We calculate the derivative of $\Gamma$, with respect to $u$, evaluated at $(0, \beta)$, which is given by
$$
 \Gamma_{u}(0,\beta) (\phi)=   (d \Delta + F(0,0)     - r(x)  - \beta ) \phi:= \mathcal{L}^{\beta} \phi, \quad \phi \in \mathbb{X}.
$$
From Theorem 2.3.20 in \cite{y2011}, we know that  the eigenvalue problem
\begin{equation}\label{1-b}
    \begin{cases}
 d \Delta u(x) + F(0,0)  u(x) - r(x)u(x)= \beta  u(x), & x\in \Omega, \\
\mathcal{B}u(x) = 0, & x\in \partial\Omega,
\end{cases}
\end{equation}
has a principal eigenvalue $\beta_*$  
with positive eigenfunction \(\varphi_*\).
This means that there exists $\beta_*$ such that  
\begin{equation}\label{beta*}
 \Gamma_{u}(0,\beta_*)\, \varphi_* = \mathcal{L}^{\beta_*} \, \varphi_*= 0, 
\end{equation}
which indicates that  
$\mathcal{L}^{\beta_*}: \mathbb{X}  \rightarrow \mathbb{Y} $ is non-invertible, 
preventing direct application of the implicit function theorem. Instead, we apply Lyapunov-Schmidt reduction by decomposing
\begin{equation}\label{xy}
    \mathbb{X}=\operatorname{Ker} (\,\mathcal{L}^{\beta_{*}} ) \oplus \mathbb{X}_1,
\quad \mathbb{Y}=\operatorname{Ker} (\,\mathcal{L}^{\beta_{*}} ) \oplus \mathbb{Y}_1,
\end{equation}
with $\operatorname{Ker} (\,\mathcal{L}^{\beta_{*}})=\operatorname{span}\{\varphi_*\} \subseteq \mathbb{X}  \subseteq \mathbb{Y}$.
 and $ \mathbb{X}_1=\left\{y \in \mathbb{X} \mid\left\langle\varphi_*, y\right\rangle=0\right\} $  
 and $\mathbb{Y}_1=\left\{y \in \mathbb{Y}: \mid\left\langle\varphi_*, y\right\rangle=0\right\}. 
$ 
Further, we summarize some properties of the operator $\mathcal{L}^{\beta_{*}}$ in the following.
\begin{proposition}
 (i) $\mathcal{L}^{\beta_*} $ is non-invertible from $ \mathbb{X}$  to $\mathbb{Y};$ 
 \\
 (ii) $\mathcal{L}^{\beta_{*}}$  is a symmetric bounded linear operator from  $\mathbb{X}$ to $\mathbb{Y};$
\\
(iii)  $(\operatorname{ Range }(\mathcal{L}^{\beta_{*}}))^{\perp}=\operatorname{Ker} (\mathcal{L}^{\beta_{*}})$ and $\mathcal{L}^{\beta_*} $ is invertible from $  \mathbb{X}_1 $  to $\mathbb{Y}_1.$
\\
(iv)  $\mathcal{L}^{\beta_{*}} $ is a Fredholm operator with index zero from  $\mathbb{X} \rightarrow \mathbb{Y};$
\end{proposition}
\begin{proof}
    (i) From \eqref{beta*}, it is easy to see that $\mathcal{L}^{\beta_*} $ is non-invertible from $ \mathbb{X}$  to $\mathbb{Y}$. 
    
    (ii) In Banach space, the continuity and boundedness of the function are equivalent, and the Laplacian operator is a symmetric operator, hence, $\mathcal{L}^{\beta_{*}}$  is a symmetric bounded linear operator from  $\mathbb{X}$ to $\mathbb{Y}$.
     
     (iii)  Since $\mathcal{L}^{\beta_{*}}$ is a symmetric operator, i.e., $  \langle \mathcal{L}^{\beta_{*}} \, x, y \rangle=\langle x, \mathcal{L}^{\beta_{*}} \,y \rangle,$ thus \((\mathcal{L}^{\beta_{*}})^*=\mathcal{L}^{\beta_{*}}\).
      For \(y\in \operatorname{Range}(\mathcal{L}^{\beta_{*}})^\perp\) and any \(x\in \mathbb{X}\) we have
  $  \langle \mathcal{L}^{\beta_{*}} x, y \rangle = 0.$ Since  $\mathcal{L}^{\beta_{*}}$ is a bounded linear operator, there exists a unique bounded linear operator $(\mathcal{L}^{\beta_{*}})^*$, which is the  adjoint operator of $\mathcal{L}^{\beta_{*}}$, such that 
  $
  0=\langle \mathcal{L}^{\beta_{*}} x, y \rangle = \langle x, (\mathcal{L}^{\beta_{*}})^* y \rangle, 
  $  
  implying that \(\langle x, (\mathcal{L}^{\beta_{*}})^* y \rangle = 0\) for all \(x\in \mathbb{X}\) and hence 
  \((\mathcal{L}^{\beta_{*}})^*y=0\), i.e., \(y \in \operatorname{ Ker }(\mathcal{L}^{\beta_{*}})^*=\operatorname{ Ker }(\mathcal{L}^{\beta_{*}})\). 
Vice versa, for \(y\in \operatorname{ Ker }(\mathcal{L}^{\beta_{*}})^*=\operatorname{ Ker }(\mathcal{L}^{\beta_{*}})\), i.e.,  \((\mathcal{L}^{\beta_{*}})^*y=0\), then for all \(x\in \mathbb{X}\), we have  $\langle \mathcal{L}^{\beta_{*}} x, y \rangle = \langle x, (\mathcal{L}^{\beta_{*}})^* y \rangle = \langle x, 0 \rangle = 0,$
  that is \(y\) is orthogonal to  \(\mathcal{L}^{\beta_{*}}x\) for all \(x\in \mathbb{X}\), i.e., \(y\in(\operatorname{Range}(\mathcal{L}^{\beta_{*}}))^\perp\). 
   Therefore we have the property 
  \begin{equation}\label{xyy}
      \operatorname{ Range }(\mathcal{L}^{\beta_{*}})^{\perp}=\operatorname{Ker} (\mathcal{L}^{\beta_{*}}).
  \end{equation}
  In addition,  $(\operatorname{ Range }(\mathcal{L}^{\beta_{*}}))^{\perp}$ is closed in $\mathbb{Y}$,  hence \begin{equation}\label{2.44}     \mathbb{Y}= \operatorname{ Range }(\mathcal{L}^{\beta_{*}}) \oplus (\operatorname{ Range }(\mathcal{L}^{\beta_{*}}))^{\perp}. \end{equation}
Consequently, \eqref{xyy} and \eqref{2.44} yield
 $$\mathbb{Y}= \operatorname{ Ker }(\mathcal{L}^{\beta_{*}}) \oplus \operatorname{ Range }( \mathcal{L}^{\beta_{*}}),$$
 and  
\begin{equation*}\label{RY1}
    \mathbb{Y}_{1}= \operatorname{ Range }( \mathcal{L}^{\beta_{*}})
\end{equation*}  from \eqref{xy}.
It is now clear that the operator $\mathcal{L}^{\beta_{*}} \mid \mathbb{X}_1: \mathbb{X}_1 \rightarrow \mathbb{Y}_1$ is bijective, and has a bounded inverse from the bounded inverse theorem, which in turn shows that  $\mathcal{L}^{\beta_{*}} \mid \mathbb{X}_1: \mathbb{X}_1 \rightarrow \mathbb{Y}_1$ is invertible.

(iv)  Since
the bounded linear operator $\mathcal{L}^{\beta_{*}}$  has finite-dimensional $\operatorname{Ker} (\,\mathcal{L}^{\beta_{*}})$, and  finite-dimensional $\operatorname{Coker} (\,\mathcal{L}^{\beta_{*}})$, i.e., $\text{codim} \operatorname{Range}(\,\mathcal{L}^{\beta_{*}})$ $=$  $\dim (\operatorname{Range}(\,\mathcal{L}^{\beta_{*}}))^{\perp}$ $=$  $\dim \operatorname{Ker} (\,\mathcal{L}^{\beta_{*}}) $ from \eqref{xyy}, $\mathcal{L}^{\beta_{*}} $ is a Fredholm operator with index zero from  $\mathbb{X} \rightarrow \mathbb{Y}.$

\end{proof}

From the proposition 2.1, we know that the operator $\mathcal{L}^{\beta_{*}}: \mathbb{X}  \rightarrow \mathbb{Y} $  is noninvertible on the entire space $\mathbb{X}$, but it is invertible on the subspace $\mathbb{X}_1$.
Now,  we can use Lyapunov-Schmidt reduction method to find solutions of \eqref{1.2}. 
For any $u \in \mathbb{X}$,  we decompose $u$ as $u=\xi+\eta$ with $\xi \in \operatorname{Ker} (\,\mathcal{L}^{\beta_{*}} )$ and $\eta \in \mathbb{X}_1$. 
From $\mathbb{Y}=\operatorname{Ker} (\,\mathcal{L}^{\beta_{*}} ) \oplus \mathbb{Y}_1,$ there exist projection operator $P$ such that $P$ from $\mathbb{Y}$ onto $\mathbb{Y}_1$ and $I-P$ from $\mathbb{Y}$ onto $\operatorname{Ker} (\,\mathcal{L}^{\beta_{*}} )$, that is, for any $ y \in \mathbb{Y},$  it can be decomposed as $y= P(y) + (I-P)(y)$ with $P(y) \in \mathbb{Y}_1  $ and $ (I-P)(y) \in \operatorname{Ker} (\,\mathcal{L}^{\beta_{*}} )$. Consequently, $\Gamma(u, \beta)=0$ is equivalent to 
\begin{equation}\label{i-p}
    P \, \Gamma(\xi+\eta, \beta)=0 
 \text{  and }  (I-P) \Gamma(\xi+\eta, \beta)=0.
\end{equation}
Obviously, $P \, \Gamma\left(0, \beta_*\right)=0$ and $P \, \Gamma_{\eta}\left(0, \beta_*\right)=\mathcal{L}^{\beta_{*}}$. Since $\mathcal{L}^{\beta_{*}} \mid \mathbb{X}_1: \mathbb{X}_1 \rightarrow \mathbb{Y}_1$ is invertible,
by applying the implicit function theorem to $P \, \Gamma(\xi+\eta, \beta)=0$, i.e., 
there exists a 
continuously  
differentiable mapping $g: U_0 \rightarrow \mathbb{X}_1$,
where $U_0$ is a neighborhood   of $(0, \beta_{*})$
in $\operatorname{Ker} (\,\mathcal{L}^{\beta_{*}} ) \times \mathbb{R}^{+}$,
 such that  
$g(0, \beta)=0$ as $\Gamma(0,\beta)=0$ for any $\beta \in \mathbb{R}^{+}$,  and
\begin{equation}\label{gxi}
   P\Gamma(\xi+g(\xi, \beta), \beta) = 0.
\end{equation} 
 Substituting $\eta=g(\xi, \beta)$ into the second equation of \eqref{i-p} yields
\begin{equation}\label{L}
    \mathcal{Q}(\xi, \beta) \equiv(I-P) \Gamma (\xi+ g(\xi, \beta), \beta)=0.
\end{equation}
Then, we replace the problem of solving system \eqref{1.2} to the problem of finding the zeros of the mapping $ \mathcal{Q}$ from $U_0$ onto 1-Dimensional space $\operatorname{Ker} (\,\mathcal{L}^{\beta_{*}})$.

For any $\xi \in \operatorname{Ker} (\,\mathcal{L}^{\beta_{*}})$, there exists $\nu \in \mathbb{R}$ such that $\xi= \nu \varphi_*(x)$. By calculating the inner product of \eqref{L} with $\varphi_*(x)$, 
we finally reduces the infinite-dimensional problem to a scalar bifurcation equation
\begin{equation*}
    f(\nu, \beta)=\int_{\Omega} \varphi_*(x) \Gamma \left(\nu \varphi_*(x)+ g\left(\nu \varphi_*(x), \beta \right), \beta\right) \mathrm{d} x=0,
\end{equation*}
where $u(x)=\nu \varphi_*(x) + g(\nu \varphi_*(x),\beta)$ and 
$g$ is obtained via the implicit function theorem.

Obviously, $f(0, \beta)=0$. To find the nonzero solution for $\nu$, expanding 
$f(\nu, \beta)$ with respect to $\nu$, 
we have
\begin{equation}\label{f(v,b)}
    f(\nu, \beta)=\nu \bar{f}(\nu,\beta) \, \text{ with }\bar{f}(\nu,\beta)= \varrho(\beta)+ \kappa(\beta) \nu+ o (\nu ),
\end{equation}
where $\varrho(\beta)=\int_{\Omega} \varphi_*\left(d \Delta+ F(0,0)  -\beta - r(x) \right) \varphi_* \d x $ and
\begin{equation}\label{kappa}
    \kappa(\beta)=  F_{1}^{\prime}(0,0) \int_{\Omega} \varphi_*^3(x) \d x +   F_{2}^{\prime}(0,0) \int_{\Omega} \int_{\Omega} \mathcal{S}(x, y) \varphi_*^2(x) \varphi_*(y) \d x \d y -\beta q^{\prime}(0) \int_{\Omega} \varphi_*^3(x) \d x,
\end{equation}
  $F_{1}^{\prime}(0,0)$ and  $F_{2}^{\prime}(0,0)$ denote the partial derivatives of $ F\left(u(x), \int_{\Omega} \mathcal{S}(x, y)  u(y)  \d y\right)$ with respect to its first   and  second component, respectively, evaluated at $(0,0)$. 
From \eqref{beta*}, $\bar{f}(0,\beta_*)=\varrho(\beta_*) = 0$ and  
\begin{equation*}
    \bar{f}_{\nu}(0,\beta_*)=  \kappa(\beta_*)= (\theta_{1}+\theta_2) - \beta_* q^{\prime}(0) \theta_3,
\end{equation*}
where
\begin{equation}\label{c1c2}
\begin{aligned}
    &\theta_{1} =F_{1}^{\prime}(0,0) \int_{\Omega} \varphi_*^3(x) \d x, \,  \theta_2= F_{2}^{\prime}(0,0) \int_{\Omega} \int_{\Omega} \mathcal{S}(x, y) \varphi_*^2(x) \varphi_*(y) \d x \d y,  \,
    \theta_3=\int_{\Omega} \varphi_*^3(x) \d x. 
    \end{aligned}
\end{equation}
If $\kappa(\beta_*) \neq 0,$ we apply the implicit function theorem to   $\bar{f}(\nu,\beta)=0$, then
there exists a neighborhood
$\mathcal{N}(\beta_*, \delta )=[\beta_*-\delta,\beta_*+\delta] \subseteq \mathbb{R}^+ $ 
such that there is
 a continuously differentiable mapping $\nu:= \nu_{\beta}$  such that $\nu_{\beta_*}=0$ and  $f\left(\nu_{\beta}, \beta\right) \equiv 0$ for all  $  \beta \in \mathcal{N}(\beta_*, \delta ) $, and hence  $u_{\beta}(x)=\nu_{\beta} \, \varphi_*(x)+  g\left(\nu_{\beta} \, \varphi_*(x), \beta \right)$ is a non-zero steady state  of \eqref{1.2}.

 Although the exact form of $u_{\beta}(x)$ cannot be obtained, we can derive an 
approximate solution under a stronger condition $\kappa(\beta) \neq 0$ in \eqref{kappa}.  In fact, from \eqref{f(v,b)},  by using Taylor expansion with respect to $\beta$ at $\beta=\beta*$,  we have
\begin{equation}\label{vb}
    v_{\beta}=-\frac{\varrho(\beta)}{\kappa(\beta)}+o(\nu )
=\frac{\theta_4}{\kappa(\beta_*)}(\beta-\beta_{*}) +o(\nu)+o(\beta-\beta_{*}),
\end{equation}
where 
\begin{equation}\label{c4}
\begin{aligned}
    &\theta_{4} = \int_{\Omega}\varphi_*^2(x) \d x.
    \end{aligned}
\end{equation}
Differentiating \eqref{gxi} with respect to $\xi$ and $\beta$ at $(0,\beta_*)$, results   
$$g_{\xi}^{\prime}(0,\beta_{*})=g_{\beta }^{\prime}(0,\beta_{*})=g_{\beta \beta}^{\prime \prime}(0,\beta_{*})=0.$$  
Therefore, the function $g(\xi, \beta)$ can be represented by
\begin{equation*}\label{gvb}
    g(\xi,\beta) \approx \frac{1}{2} g_{\xi \xi}^{\prime \prime}(0,\beta_{*}) \xi^2 + g_{\xi \beta}^{\prime \prime}(0,\beta_{*}) \xi (\beta-\beta_{*}) = m(x) \frac{\theta_4 \varphi_*(x)}{\kappa(\beta_*)} (\beta-\beta_{*})^2,
\end{equation*}
where $m(x)=  \frac{1}{2} g_{\xi \xi}^{\prime \prime}(0,\beta_{*})  \frac{\theta_4  \varphi_* (x)}{\kappa (\beta_*)} + g_{\xi \beta}^{\prime \prime}(0,\beta_{*})$, due to $g(0,\beta)=0$ for any $\beta \in \mathbb{R^+}$ and \eqref{vb}. Consequently, the solution $u_{\beta}(x)$ has the form of
\begin{equation}\label{ux}
    u_{\beta}(x) \approx \frac{\theta_4 \varphi_*(x)}{\kappa(\beta_*)} \left( (\beta-\beta_{*})  +  m(x) (\beta-\beta_{*})^2 \right).
\end{equation}

Further we can determine the sign of $u_{\beta}(x)$.
Since $\beta_*$ is the principal eigenvalue of system \eqref{1-b}, the corresponding eigenfunction $\varphi_*(x)$ is positive on $\Omega$. 
It follows that
when $\kappa(\beta_*)<0$ and $\beta<\beta_*$ (or $\kappa(\beta_*)>0$ and $\beta>\beta_*$), the solution
$u_{\beta}(x)$ remains
positive in 
$\mathcal{N}_{-}(\beta_*, \delta )=[\beta_*-\delta,\beta_*)$ (or $\mathcal{N}_{+}(\beta_*, \delta )=(\beta_*, \beta_* + \delta]$).
Hence, we obtain the following result.

\begin{theorem}\label{ubx}
    Suppose that $\kappa(\beta_*) \neq 0$.
    Then there exist 
    $\delta >0$ and a continuously differentiable mapping $\nu_\beta: \mathcal{N}(\beta_*, \delta ) \rightarrow \mathbb{R}$
    such that system \eqref{1.2} admits a steady-state solution  $u_{\beta}(x)$   for all $ \beta \in \mathcal{N}(\beta_*, \delta ).$ 
Moreover,  if $\kappa(\beta) \neq 0$, then  $u_{\beta}(x)$ can be expressed in the form given in \eqref{ux}, and it is positive when $\kappa(\beta_*) <0 $ with $ \beta \in \mathcal{N}_{-}(\beta_*, \delta )$ (or  $ \kappa(\beta_*) >0$ with $ \beta \in \mathcal{N}_{+}(\beta_*, \delta )).$ 
\end{theorem}

\begin{remark}
In contrast to earlier studies \cite{css2016,wh2020,lyq2023} that obtained  nonhomogeneous steady states based on particular solutions, the present work establishes a general bifurcation framework and derives an approximate representation of the positive steady state near $\beta_*$.
\end{remark}

\section{Stability Analysis}

In this section, we explore the local stability of the system \eqref{1.1}
at its steady-state solutions.
We first linearize system \eqref{1.1} around a steady state 
$u(x)$. The resulting linearized system is 
\begin{equation}\label{3.2}
    \begin{aligned}
&\left\{\begin{aligned}
 &u_{t}(x, t)=  d \Delta u(x,t) - \left[
 r(x) +  \beta q(u(x)) +  \beta q^{\prime}(u(x))u(x)
 \right] u(x, t) +   A(x)  u(x) u(x, t)
 &  \\
 &   
 \qquad + u(x, t) F\left(u(x), \int_{\Omega} \mathcal{S}(x, y) u(y) \d y\right) +   B(x)  u(x) \int_{\Omega} \mathcal{S}(x, y) u(y, t-\tau) \d y,  \, x \in \Omega, t>0, \\
&\mathcal{B}u(x) = 0, \, x \in \partial \Omega, t>0,
\end{aligned}\right.\\
\end{aligned}
\end{equation}
where $$A(x)=F^{\prime}_{1}\left(u(x), \int_{\Omega} \mathcal{S}(x, y) u(y) \d y\right) , \ \ \text { and } B(x)=F_{2}^{\prime}\left(u(x), \int_{\Omega} \mathcal{S}(x, y) u(y) \d y\right).$$ 
Following the framework in Wu \cite{wu1996}, the semigroup generated by the solution  of \eqref{3.2} admits the infinitesimal generator $\mathscr{A}_{\tau, \beta}$ defined by
$$
\mathscr{A}_{\tau, \beta} \,\Phi =\dot{\Phi}, \ \Phi \in \operatorname{Dom}\left(\mathscr{A}_{\tau, \beta}\right),$$
with domain
$$
\begin{aligned}
&\operatorname{Dom}\left(\mathscr{A}_{\tau, \beta}\right)  =\left\{\Phi \in  C^1 ([-\tau,0],\mathbb{Y_{\mathbb{C}}}) \mid \Phi(0) \in \mathbb{X}_{\mathbb{C}}, \, \dot{\Phi}(0) = [ d \Delta 
 - r(x) -  \beta q(u(x)) -  \beta q^{\prime}(u(x))u(x)] \Phi(0) 
\right. \\
 &\left.
\qquad  +\left[ F\left(u(x), \int_{\Omega} \mathcal{S}(x, y) u(y) \d y \right) + A(x) u(x)  \right] \Phi(0)  
 +   B(x)  u(x) \int_{\Omega} \mathcal{S}(x, y) \Phi(-\tau)(y) \d y
 \right\}.
\end{aligned}
$$
 The spectral set of $\mathscr{A}_{\tau, \beta}$ is
$$
\sigma\left(\mathscr{A}_{\tau, \beta}\right)=\left\{\lambda \in \mathbb{C} \mid \Pi(\tau, \beta, \lambda  ) \phi(x)=0 \text { for some } \phi(x) \in \mathbb{X}_{\mathbb{C}} \backslash \{0\} \right\},
$$
where
\begin{equation}\label{lambda}
    \begin{aligned}
\Pi(\tau, \beta, \lambda) \phi:= & 
 [ d \Delta  
 - r(x) -  \beta q(u(x)) -  \beta q^{\prime}(u(x))u(x)] \phi(x) \\
& + \left[ F\left(u(x), \int_{\Omega} \mathcal{S}(x, y) u(y) \d y \right) + A(x) u(x)  \right] \phi(x)  \\
& +  B(x)  u(x) e^{-\tau \lambda} \int_{\Omega} \mathcal{S}(x, y) \phi(y) \d y -\lambda \phi(x) =0.
\end{aligned}
\end{equation}

At trivial steady state $(0,\beta)$, for all $\beta \in \mathbb{R}^{+}$, the characteristic equation \eqref{lambda} becomes
\begin{equation}\label{CE0}
    \begin{aligned}
[d \Delta  
 - r(x) + F(0,0)  - \beta] \phi(x) =\lambda  \phi(x).
\end{aligned}
\end{equation}
From \eqref{beta*}, it is easy to see that $\lambda = \beta_* - \beta,$ 
and hence, when $\beta > \beta_*, $ the trivial steady state is locally asymptotically stable; when $\beta < \beta_*, $ the trivial steady state is unstable.

From a biological standpoint, negative steady states are not meaningful; therefore, we restrict our attention to positive steady-state solutions $u_{\beta}(x)$.
Next, we establish the conditions ensuring the local asymptotic stability of the positive steady-state solution $u_{\beta}(x)$, that is, when the spectrum $\sigma\left(\mathscr{A}_{\tau, \beta}\right)$ consists solely 
of eigenvalues with negative real part. Here, we assume that $\lambda_{\beta}$ is the eigenvalue of characteristic equation \eqref{lambda} with the corresponding eigenfunction $ \phi_\beta \in$ $\mathbb{X}_{\mathbb{C}}$, normalized such that $\left\|\phi_\beta\right\|_{X_{\mathbb{C}}}=1.$
We now present the following lemma.

\begin{lemma}
    Suppose $\kappa(\beta) \neq 0$. If $\Pi(\tau_{\beta}, \beta , \lambda_{\beta}) \phi_{\beta}=0$ with $Re (\lambda_{\beta}) \geq 0$, then $\left|\frac{ \kappa(\beta_*) \lambda_{\beta}  }{\beta-\beta_*}\right|$ is bounded for $\beta \in \mathcal{N}(\beta_*, \delta )$.
\end{lemma}

\begin{proof}
Denote the operator $\mathcal{L}$ from $\mathbb{X}$ to $\mathbb{Y}$ as
$$\mathcal{L}= d \Delta  -  \beta q(u_{\beta}(x)) - r(x)+   F\left( u_{\beta}(x), \int_{\Omega} \mathcal{S}(\cdot, y) u_{\beta}(y) \d y\right).$$ 
From \eqref{1.2}, obviously, $\mathcal{L} u_{\beta}(x)=0$, which means that $0$ is the principal eigenvalue of $\mathcal{L}$ with principal eigenfunction $u_{\beta}(x)$, 
and hence $\left\langle \mathcal{L} \phi_\beta, \phi_\beta \right\rangle < 0$.
From \eqref{lambda}, we have
$$
\left\langle \mathcal{L}  \phi_\beta - \beta q^{\prime}(u_{\beta} )u_{\beta}  \phi_\beta +   A(x) u_\beta  \phi_\beta +    B(x) u_\beta  e^{- \tau_{\beta} \lambda_{\beta} } \int_{\Omega} \mathcal{S}(x, y) \phi_\beta(y) \d y -\lambda_{\beta} \phi_\beta, \phi_\beta \right\rangle=0.
$$
then
\begin{equation}\label{lab}
    \lambda_{\beta} <
\left\langle  ( A(x)-\beta q^{\prime}(u_{\beta} )) u_\beta(x) \phi_\beta, \phi_\beta \right\rangle +\left\langle   B(x) u_\beta(x) \,  e^{- \tau_\beta \lambda_{\beta}} \int_{\Omega} \mathcal{S}(x, y) \phi_\beta(y) \d y, \phi_\beta \right\rangle
\end{equation}
from $\left\langle \mathcal{L} \phi_\beta, \phi_\beta \right\rangle < 0$. 
Using the expression of $u_{\beta}(x)$   in \eqref{ux}
together with the calculation in \eqref{lab}, we have
 $$
\begin{aligned}
    \left|  \frac{ \kappa(\beta_*) \lambda_{\beta}  }{\beta-\beta_*}  \right|  
< &
  \left|\langle ( A(x)-\beta q^{\prime}(u_{\beta} ))   [ 1 + |\beta_{*}-\beta| m(x) ] \theta_4 \varphi_* 
\phi_{\beta}, \phi_{\beta} \rangle \right| \\
& +   \left|  \left\langle   B(x) 
\theta_4 \varphi_* [ 1 + |\beta_{*}-\beta| m(x) ]
e^{- \tau_{\beta} \lambda_{\beta} }   \int_{\Omega} \mathcal{S}(x, y) \phi_\beta(y) \d y \, \phi_{\beta}, \phi_{\beta} \right\rangle \right| \\
 \leq  &
\theta_4 \left\| \varphi_*(x) \right\|_{\infty} 
\left( 
1
+ |\beta_{*}-\beta|\left\|m(x)\right\|_{\infty}
\right)
\left( \left\| A(x)-\beta q^{\prime}(u_{\beta} )\right\|_{\infty}
+ \left\|  B(x)\right\|_{\infty} \left(\max _{\bar{\Omega} \times \bar{\Omega}} \mathcal{S}(x, y)\right)|\Omega| \right).
\end{aligned}
$$ 
By the continuous differentiability of $F(u,v)$ and $q(u)$, it follows that $A(x)$, $B(x)$ and $q^{\prime}(u_{\beta})$ are bounded in $X_{\mathbb{C}}$. Furthermore, by the
embedding theorem \cite{sobolev}, 
there exists $M >0$ such that 
$$\left\| A(x)- \beta q^{\prime}(u_{\beta})\right\|_{\infty}<M, \quad \text{and} \  \left\| B(x)\right\|_{\infty}<M.$$ 
Consequently, 
$\left|\frac{ \kappa(\beta_*) \lambda_{\beta}  }{\beta-\beta_*}\right|$ remains bounded for $\beta \in \mathcal{N}(\beta_*, \delta )$.
\end{proof}

The following theorem establishes the local stability of the positive steady-state solution.
\begin{theorem}\label{th3.2}
     If $ \kappa(\beta) \neq 0 $, and  
     $$
    (\mathrm{A}_1):  \, \kappa(\beta_*)<0 \text { and } \tilde{\kappa}(\beta_*)<0,  
     $$
 where $\tilde{\kappa}(\beta_*) = \theta_{1}-\theta_2 - \beta_* q^{\prime}(0) \theta_3$ with parameters $\theta_{i} $ $(i=1,2,3)$ defined in \eqref{c1c2},
then there exists a constant $\delta_1$ $(0<\delta_1 \leq \delta)$, such that for $ \beta \in \mathcal{N}_{-}(\beta_*, \delta_1 )$,  every eigenvalues $\lambda \in \sigma\left(\mathscr{A}_{\tau, \beta}\right)$
 has negative real part  for $\tau \geq 0$. Consequently, the positive steady-state solution $u_{\beta}(x)$ is locally asymptotically stable.
\end{theorem}
\begin{proof}
   By contradiction, assuming   there exists a sequence $\left\{\beta_n\right\}_{n=1}^{\infty} \subset \mathcal{N}_{-}(\beta_*, \delta_1 )$ such that 
   $\left\{\tau_{\beta_n},\lambda_{\beta_n},\phi_{\beta_n}\right\}_{n=1}^{\infty} $ solves characteristic equation $ \Pi \left(\tau_{\beta_n},\beta_n,\lambda_{\beta_n}\right) \phi_{\beta_n}=0$ with
   $Re(\lambda_{\beta_{n}}) \geq 0$, $\tau_{\beta_n} \geq 0,$ and $ \lim _{n \rightarrow \infty} \beta_n=\beta_*$. 
 Ignoring a scalar factor, we suppose that $\|\phi_{\beta_{n}}\|_{Y_{\mathbb{C}}}^2=\|\varphi_{*}\|_{Y_{\mathbb{C}}}^2.$
   Note that $ \mathbb{X}_{\mathbb{C}}=\operatorname{Ker} (\,\mathcal{L}^{\beta_{*}} )_{\mathbb{C}} \oplus (\mathbb{X}_1)_{\mathbb{C}},$
in the light of Lemma 3.1 in \cite{LJ2022JMAA},
 for $\phi_{\beta_{n}} \in \mathbb{X}_{\mathbb{C}}\backslash \{0\}$, $\phi_{\beta_{n}}$ can be represented as   
\begin{equation}\label{3.3}
\phi_{\beta_n}=\alpha_{\beta_n} \varphi_{*}(x)+\left(\beta_n -\beta_* \right) {z}_{\beta_{n}}(x), \,\alpha_{\beta_n} \in \mathbb{C},  {z}_{\beta_{n}} \in (\mathbb{X}_1)_{\mathbb{C}},
\end{equation}
and hence,
\begin{equation*}\label{3.33}
\left\|\phi_{\beta_n}\right\|_{Y_{\mathbb{C}}}^2
=\left\langle \phi_{\beta_n}, \phi_{\beta_n}\right\rangle
=\alpha_{\beta_n}^2\|\varphi_{*}\|_{Y_{\mathbb{C}}}^2+ \left(\beta_* -\beta_n \right)^2\left\|{z}_{\beta_{n}}\right\|_{Y_{\mathbb{C}}}^2=\|\varphi_{*}\|_{Y_{\mathbb{C}}}^2.
\end{equation*}
Consequently, $|\alpha_{\beta_{n}}| \leq 1$  and 
$\lim _{n \rightarrow \infty} \alpha_{\beta_{n}}=1 $.
By Lemma 3.2, since $\left|\frac{ \kappa(\beta_*) \lambda_{\beta}  }{\beta-\beta_*}\right|$ is bounded for $\beta \in \mathcal{N}(\beta_*, \delta_1 )$, we can write 
$\lambda_{\beta_n}$ as 
 \begin{equation}\label{l}
\lambda_{\beta_n}=\frac{l_{\beta_n}}{\kappa(\beta_{*})}\left(\beta_n-\beta_* \right),
  \end{equation}
 with $Re(l_{\beta_{n}}) \geq 0$. 
 Substituting 
$u_{\beta_{n}}(x)=\frac{\theta_4}{\kappa(\beta_{*})}(\beta_{n}-\beta_{*}) \varphi_*(x) ( 1 +
 (\beta_{n}-\beta_{*}) 
m(x) ),  $  Eqs. \eqref{3.3} and \eqref{l} into  $ \Pi  (\tau_{\beta_n},\beta_n,\lambda_{\beta_n}) \phi_{\beta_n} =0,$ we can see that 
 $\left(  {z}_{\beta_{n}} , \alpha_{\beta_n}, l_{\beta_n}, \tau_{\beta_n}, \beta_{n} \right)$ satisfies the following equation
\begin{equation}\label{H1}
\begin{aligned}
&\left(d \Delta+ F(0,0)  - r(x)-\beta_*\right) z_{\beta_{n}} 
+ 
\left(
1  - \beta_n \bar{u}_{\beta_{n}}( q^{\prime}(0)+  q^{\prime}(u_{\beta_n} ))- \frac{l_{\beta_{n}} }{\kappa(\beta_*)}
\right) \phi_{\beta_{n}}
 \\
& +    
\left( F_{1}^{\prime}(0,0) \bar{u}_{\beta_{n}}  +F_{2}^{\prime}(0,0) 
\int_{\Omega} \mathcal{S}(\cdot, y) \bar{u}_{\beta_{n}} (y) \d y
   + A(x) \bar{u}_{\beta_{n}}
\right) \phi_{\beta_{n}}
 \\
& +     B(x) \bar{u}_{\beta_{n}}
\int_{\Omega} \mathcal{S}(\cdot, y) \phi_{\beta_{n}}(y) \d y \, e^{- \tau_{\beta_{n}} \frac{l_{\beta_{n}}}{\kappa(\beta_*)} \left( \beta_{n}-\beta_{*}\right)}  =0,
\end{aligned}  
\end{equation}
where $\bar{u}_{\beta_{n}}(x)=\frac{\theta_4}{\kappa(\beta_{*})} \varphi_*(x) ( 1 +
 (\beta_{n}-\beta_{*}) 
m(x) ).$

Next, we address the convergence of sequence 
\begin{equation}\label{xl3.7}
    \left\{
\left(  z_{\beta_n},\alpha_{\beta_n},l_{\beta_n},  e^{- \frac{\tau_{\beta_n}}{\kappa(\beta_*)} Re(l_{\beta_n} )\left( \beta_{n}- \beta_{*}\right)},e^{-\mathrm{i} \frac{\tau_{\beta_n}}{\kappa(\beta_*)} Im(l_{\beta_n} )\left(\beta_{n}- \beta_{*}\right)} \right)\right\}_{n=1}^{\infty}.
\end{equation}
Let $\beta_2$ be the second eigenvalue of system \eqref{1-b}, then for any $\psi$ satisfies $\left\langle \varphi_{*}, \psi \right\rangle=0$, we have 
$$
\begin{aligned}
\left|\left\langle\left(d\Delta+F(0,0) -r(x)-\beta_*\right) \psi, \psi \right\rangle\right|  
 &=
\left|
\langle -(d\Delta+F(0,0) -r(x)) \psi , \psi \rangle 
+  \langle   \beta_*  \psi , \psi \rangle
\right| \\
&\geq 
\left|
\langle  -\beta_2 \psi , \psi \rangle 
+  \langle  \beta_* \psi , \psi \rangle
\right|= |\beta_* - \beta_2| \left\| \psi \right\|_{Y_{\mathbb{C}}}^2. 
\end{aligned}
$$
Hence, by taking the inner product of  \eqref{H1} with $z_{\beta_{n}}$, we obtain
$$
\begin{aligned}
\left\|z_{\beta_{n}}\right\|_{Y_{\mathbb{C}}}^2 \leq & \, \frac{1}{|\beta_* - \beta_2|} \left|\left\langle\left(d \Delta+ F(0,0)  - r(x)-\beta_*\right) z_{\beta_n}, z_{\beta_n}\right\rangle\right| \\
= &  \,  \frac{1}{|\beta_* - \beta_2|}  \left|\left\langle  
     F_{2}^{\prime}(0,0) \phi_{\beta_{n}}
\int_{\Omega} \mathcal{S}(\cdot, y) \bar{u}_{\beta_{n}} (y) \d y 
    +B(x) \bar{u}_{\beta_{n}}
\int_{\Omega} \mathcal{S}(\cdot, y) \phi_{\beta_{n}}(y) \d y \, e^{-  \frac{\tau_{\beta_{n}} l_{\beta_n}}{\kappa(\beta_{*})} \left( \beta_{n}-\beta_{*}\right)}, z_{\beta_n}\right\rangle\right| \\
& +  
\frac{1}{|\beta_* - \beta_2|}  \left|\left\langle \left(
1 - \frac{l_{\beta_n}}{\kappa(\beta_{*})} -  \bar{u}_{\beta_{n}}  \left(\beta_n q^{\prime}(0)+\beta_n q^{\prime}(u_{\beta_{n}} ) + F_{1}^{\prime}(0,0)+ A(x) \right) 
\right) \phi_{\beta_{n}}  , z_{\beta_n}\right\rangle\right|
\\
\leq &  \, \frac{1}{|\beta_* - \beta_2|}  
 \left\| 1 - \frac{l_{\beta_n}}{\kappa(\beta_{*})} -  \bar{u}_{\beta_{n}}  \left(\beta_n q^{\prime}(0)+\beta_n q^{\prime}(u_{\beta_{n}} ) + F_{1}^{\prime}(0,0)+ A(x) \right)   \right\|_{\infty}
\left|\left\langle  \phi_{\beta_{n}}  , z_{\beta_n}\right\rangle\right| \\
& +   \frac{1}{|\beta_* - \beta_2|}  
\left\|
    F_{2}^{\prime}(0,0) 
\right\|_{\infty}  
\left\|
 \bar{u}_{\beta_{n}}(x) 
\right\|_{\infty} |\Omega|
  \max _{\bar{\Omega} \times \bar{\Omega}} \mathcal{S}(x, y) \left|\left\langle  \phi_{\beta_{n}}  , z_{\beta_n}\right\rangle\right|\\
& +  \frac{1}{|\beta_* - \beta_2|}  \left\|     B(x) \bar{u}_{\beta_{n}}
 e^{-  \frac{\tau_{\beta_{n}} l_{\beta_n}}{\kappa(\beta_{*})} \left( \beta_{n}-\beta_{*}\right)} \right\|_{\infty} |\Omega|  \max _{\bar{\Omega} \times \bar{\Omega}} \mathcal{S}(x, y) \left|\left\langle  \phi_{\beta_{n}}  , z_{\beta_n}\right\rangle\right|\\
\leq  & \,
3\frac{\tilde{M} \|\varphi_{*}\|_{Y_{\mathbb{C}}} }{|\beta_* - \beta_2|} 
\left\|z_{\beta_n}\right\|_{Y_{\mathbb{C}}}
+
3\frac{ \tilde{M} \left|\beta_*-\beta_n \right|}{|\beta_* - \beta_2|}   \left\|z_{\beta_n}\right\|_{Y_{\mathbb{C}}}^2,
\end{aligned}
$$
for constant $\tilde{M}>0$.
It then follows that $\left\{z_{\beta_n}\right\}_{n=1}^{\infty}$ is bounded in $Y_{\mathbb{C}}$. Since the operator $d \Delta+ F(0,0)  - r(x)-\beta_*:\left(\mathbb{X}_1\right)_{\mathbb{C}}\mapsto\left(\mathbb{Y}_1\right)_{\mathbb{C}}$ has a bounded inverse, by applying $\left(d \Delta+ F(0,0)  - r(x)-\beta_*\right)^{-1}$ on \eqref{H1}, we know that $\left\{z_{\beta_n}\right\}_{n=1}^{\infty}$ is also bounded in $\mathbb{X}_{\mathbb{C}}$.

Due to the fact that  $\mathbb{X}_{\mathbb{C}}$ is compactly embedded into $\mathbb{Y}_{\mathbb{C}}$,
the sequence \eqref{xl3.7}
is precompact in $\mathbb{Y}_{\mathbb{C}} \times \mathbb{R}^3 \times \mathbb{C}$.
Then, there is a subsequence of \eqref{xl3.7}
convergent to $\left(z_*, \alpha_*, l_*,  \omega_1, e^{-\mathrm{i} \omega_2}\right)$   in $\mathbb{Y}_{\mathbb{C}} \times \mathbb{R}^3 \times \mathbb{C}$, 
and satisfies \eqref{H1}, that is
\begin{equation}\label{reim}
    \begin{aligned}
  &\left( d \Delta+ F(0,0)  - r(x)-\beta_* \right)  z_{*} +
 \varphi_* \left(
1 - 2\beta_* q^{\prime}(0)\frac{\theta_4 \varphi_*}{\kappa(\beta_*)}     - \frac{l_{*}}{\kappa(\beta_*)}  
\right)
  \\
 &+ \frac{\theta_4  \varphi_* }{\kappa(\beta_*)}
 \left(
2 F_{1}^{\prime}(0,0)  \varphi_*   +
F_{2}^{\prime}(0,0) \int_{\Omega} \mathcal{S}(\cdot, y) \varphi_{*}(y)    \d y
\right)
\\
 &
+ \frac{\theta_4 \varphi_*  }{\kappa(\beta_*)}  F_{2}^{\prime}(0,0)     \omega_1 e^{- \mathrm{i} \omega_2}  
\int_{\Omega} \mathcal{S}(\cdot, y) \varphi_{*}(y)    \d y  =0,\\
\end{aligned}
\end{equation}
where
$$
 z_*  \in \mathbb{Y}_{\mathbb{C}}, \quad \alpha_*=1, \quad
 l_* \in \mathbb{C}\left(Re(l_*)  \geq 0\right),  \quad  \omega_1 \in (0,1]  \text { and }    \omega_2 \in[0,2 \pi).
$$
Further we have $z_{*} \in \left(\mathbb{X}_1\right)_{\mathbb{C}}$ as $\left(\mathbb{X}_1\right)_{\mathbb{C}}$ is closed. 
Consequently, by taking the inner product of $\varphi_*(x)$
with both sides of \eqref{reim}, we obtain
$$
\theta_1  + \theta_2  \omega_1 e^{-i \omega_2} -\beta_* q^{\prime}(0)\theta_3 
  -     l_*=0.
$$
Separating the real and imaginary parts of above equation yields
$$
\left\{\begin{array}{l}
Re:     Re(l_{*}) =   \theta_1 +    \theta_2  \omega_1 \cos(\omega_{2}) - \beta_* q^{\prime}(0)\theta_3, \\
Im:     \omega_1 \theta_2  \sin(\omega_{2})+ Im(l_* )  =0.
\end{array}\right.
$$
Under condition $(\mathrm{A}_1)$, we have
$$
  Re(l_{*})=  \theta_1 +    \theta_2  \omega_1 \cos(\omega_{2}) - \beta_* q^{\prime}(0)\theta_3 < 0,
$$
which contradicts 
  $Re(l_{*}) \geq 0.$
Therefore,  all the  eigenvalues $\lambda \in \sigma\left(\mathscr{A}_{\tau, \beta}\right)$
 have negative real part for $\tau \geq 0$, and the positive steady-state solution $u_{\beta}(x)$ is locally asymptotically stable for $ \beta \in \mathcal{N}_{-}(\beta_*, \delta_1 )$.
\end{proof}

\begin{remark}\label{remark3.3}

For the case where another positive steady state exists, namely when $ \kappa(\beta_*)>0$ \text { and }  $ \beta \in \mathcal{N}_{+}(\beta_*, \delta )$, 
we adopt the sequence-convergence argument used in Theorem \ref{th3.2}. 
Assume that there exists a sequence $\left\{\beta_n\right\}_{n=1}^{\infty} \subset \mathcal{N}_{+}(\beta_*, \delta_1 )$ with  $ \lim _{n \rightarrow \infty} \beta_n=\beta_*,$ such that 
   $\left\{\tau_{\beta_n},\lambda_{\beta_n},\phi_{\beta_n}\right\}_{n=1}^{\infty} $ satisflies the characteristic equation \eqref{lambda}, without imposing any prior assumption on $Re (\lambda_{\beta_n})$. Following the same line of reasoning as before, we obtain that if
   $  \kappa(\beta_*)>0 \text { and } \tilde{\kappa}(\beta_*)>0,$ then
   $$
  Re(l_{*})=  \theta_1 +    \theta_2  \omega_1 \cos(\omega_{2}) - \beta_* q^{\prime}(0)\theta_3 > 0,
$$
which implies that there exists at least one eigenvalue $\lambda$ with a positive real part. Consequently, the positive steady-state solution $u_{\beta}(x)$ is unstable for  $\tau \geq 0$.

\end{remark}

\section{Hopf bifurcation}

We first rule out the possibility of a 
Hopf bifurcation at the trivial steady state $u_{\beta_*}(x) = 0$ based on the characteristic equation \eqref{CE0}.
Hence, our subsequent analysis concentrates on proving the existence of a Hopf bifurcation at the positive steady-state solution
$u_\beta(x)$, established
in Theorem \ref{ubx} for $ \beta \in  \mathcal{N}_{-}(\beta_*, \delta_1 ) $ or $ \mathcal{N}_{+}(\beta_*, \delta_1 )$.

As we know, 
$\mathrm{i} \omega \in \sigma\left(\mathscr{A}_{\tau, \beta}\right)$ $(\omega >0)$ for some $\tau \geq 0$ if and only if 
\begin{equation}\label{lambdaiw}
    \begin{aligned}
 [ d \Delta  
 - r(x) -    \beta q(u_{\beta}(x))  -  \beta q^{\prime}(u_{\beta}(x))u_{\beta}(x)] \phi(x)  +  B(x)  u_{\beta}(x) e^{-\mathrm{i} \theta  } \int_{\Omega} \mathcal{S}(x, y) \phi(y) \d y -\mathrm{i} \omega \phi(x)\\
 +\left[ F\left(u_{\beta}(x), \int_{\Omega} \mathcal{S}(x, y) u_{\beta}(y) \d y \right) +A(x) u_{\beta}(x)  \right] \phi(x)  =0
\end{aligned}
\end{equation}
is solvable for a pair of $(\omega, \theta)$, where $\theta= \omega \tau \geq 0, $ and  $ \phi \in \mathbb{X}_{\mathbb{C}} \backslash\{0\}$. 
Regarding to the form of $u_{\beta}(x)$ in \eqref{ux}, we represent $\omega$ as
\begin{equation}\label{op}
 \omega =\frac{l }{\kappa(\beta_*)} \left(\beta -\beta_* \right),   \ l >0.  
\end{equation}
The function  $ \phi \in \mathbb{X}_{\mathbb{C}} \backslash\{0\}$ can be represented by  
\begin{equation}\label{phi4.3}
\phi(x)=\alpha  \varphi_{*}(x)+\left(\beta  -\beta_* \right) {z} (x), \,\alpha  \in \mathbb{C},  {z}  \in (\mathbb{X}_1)_{\mathbb{C}}, \text{ with }  \|\phi\|_{Y_{\mathbb{C}}}^2=\|\varphi_{*}\|_{Y_{\mathbb{C}}}^2. 
\end{equation}
Consequently,
\begin{equation}\label{phi4.4}
    h_1 \left(\alpha, z,\beta \right) = \left(\alpha ^2-1\right)\|\varphi_{*}\|_{Y_{\mathbb{C}}}^2+\left(\beta  -\beta_{*}\right)^2\| z  \|_{Y_{\mathbb{C}}}^2=0.
\end{equation}
 Substituting    Eqs. \eqref{op}, \eqref{phi4.3} and $u_{\beta}(x)$ in \eqref{ux} into characteristic equation \eqref{lambdaiw}, we have
\begin{equation} \label{3.14}
    \begin{aligned}
  h_2\left(\theta, \alpha, l,z, \beta \right) = &\left( d \Delta+ F(0,0)  - r(x)-\beta  \right)  z - \left( \beta  \bar{u}_{\beta}  (q^{\prime}(0) + q^{\prime}(u_{\beta}) ) + \frac{il   }{\kappa(\beta_*)} \right) [\alpha  \varphi_{*}(x)+ \left( \beta  -\beta_* \right) {z} (x) ] \\
  & + \left( F_{1}^{\prime}(0,0)  \bar{u}_{\beta}  +
F_{2}^{\prime}(0,0) \int_{\Omega} \mathcal{S}(\cdot, y) \bar{u}_{\beta}(y) \d y + A(x)\bar{u}_{\beta} \right) [ \alpha  \varphi_{*}(x)+ \left( \beta  -\beta_* \right) {z} (x) ] \\
 & + B(x) \bar{u}_{\beta} e^{- \mathrm{i} \theta} \int_{\Omega} \mathcal{S}(\cdot, y) [ \alpha  \varphi_{*}(y)+ \left( \beta  -\beta_* \right) {z} (y) ]    \d y
=0.\\
\end{aligned}
\end{equation}
Define $ H \left(\theta, \alpha, l,z, \beta \right) :  [0,2\pi) \times \mathbb{R}^2 \times (\mathbb{X}_1)_{\mathbb{C}}  \times \mathbb{R} \longrightarrow (\mathbb{Y} )_{\mathbb{C}} \times \mathbb{R} $ by $ H=(h_1,h_2).$ 
Our purpose is to find the zeros of $H$ for $\beta$ near $\beta_*$. We firstly show that the existence of the zeros of $H$ when $\beta= \beta_*.$ Clearly, $\alpha=\alpha_{\beta_*}=1$ from \eqref{phi4.4}.
Multiplying both sides of \eqref{3.14} by $\varphi_*(x)$ and integrating over $\Omega$ at $\beta = \beta_*$,
we obtain that \eqref{3.14} is
solvable if and only if 
\begin{equation}\label{ri}
    \left\{\begin{array}{l}
Re:  \theta_1+  \theta_2 \cos \theta  - \beta_* q^{\prime}(0)\theta_3=0, \\
Im:\theta_2   \sin \theta  +l  =0.
\end{array}\right.
\end{equation}
Under the condition $$ (\mathrm{A}_2): \kappa(\beta_*)<0 \text { and } \tilde{\kappa}( \beta_* )>0, $$ 
it is straightforward to see that
$ \theta_2<0$. Then, from the imaginary part in \eqref{ri}, we know 
$\sin \theta>0$. Solving \eqref{ri} yields 
\begin{equation}\label{theta12}
\begin{aligned}
   \theta&=  \theta_{\beta_*} 
     =    \arccos \left( \frac{ \beta_* q^{\prime}(0) \theta_3 - \theta_1}{ \theta_2} \right) \in (0, \pi),  \\
   l&=   l_{\beta_*}   =     \sqrt{-[  \theta_{1}+\theta_2 - \beta_* q^{\prime}(0) \theta_3][  \theta_{1}-\theta_2 - \beta_* q^{\prime}(0) \theta_3]} =   \sqrt{ - \kappa(\beta_*)\tilde{\kappa}( \beta_*)}.
\end{aligned}
\end{equation}
Consequently, $z=z_{\beta_*}$ is determined from
\eqref{3.14}.
Therefore, we know that  $H \left(\theta, \alpha, l,z, \beta_* \right)=0$ has a unique solution $\left(\theta_{\beta_*}, 1, l_{\beta_*}, z_{\beta_*} \right).$
Now, we show that the existence of solutions of $H \left(\theta, \alpha, l,z, \beta \right)=0$ for $ \beta \in  \mathcal{N}_{-}(\beta_*, \delta_2 ),$ and we have the following results.

\begin{theorem}\label{th4.3}
    Assume that the condition $ (\mathrm{A}_2) $   holds, then there exist a constant $\delta_2$ and a continuously differentiable mapping $\beta \longmapsto \left(\theta_{\beta}, \alpha_{\beta},   l_{\beta}, z_{\beta} \right) \in [0,2\pi) \times \mathbb{R}^2 \times (\mathbb{X}_1)_{\mathbb{C}}$ such that 
$$
H(\theta_{\beta}, \alpha_{\beta},   l_{\beta}, z_{\beta} )=0
$$
  for  $ \beta \in \mathcal{N}_{-}(\beta_*, \delta_2 ),$  $0<\delta_2 \leq \delta_1$. 
  Further, for each fixed  $ \beta \in \mathcal{N}_{-}(\beta_*, \delta_2 ),$
  $\lambda=i \omega (\omega>0) \in   \sigma\left(\mathscr{A}_{\tau, \beta}\right)$ if and only if   
  $$   \omega =   \omega_{\beta} = \frac{ l_\beta}{\kappa(\beta_*)}\left(\beta-\beta_* \right), \quad \tau=\tau_{k}=\frac{\theta_{\beta}+ 2k \pi}{\omega_{\beta}}, \, k \in \mathbb{N}_0. 
   $$
and $\phi= \phi_{\beta}= \alpha_{\beta}  \varphi_{*}(x)+\left(\beta  -\beta_* \right) z_{\beta}.$ 
\end{theorem}

\begin{proof}
From above analysis, we know $H \left(\theta_{\beta_*}, \alpha_{\beta*}, l_{\beta_*}, z_{\beta_*} \right)=0.$ 
To establish the existence of a solution to 
$$H \left(\theta, \alpha, l,z, \beta  \right)=0 \ \text{for} \  \beta \in \mathcal{N}_{-}(\beta_*, \delta_2 ),$$ 
 via the implicit function theorem, we compute the Fréchet derivative of $H $ with respect to $(\theta, \alpha,   l, z)$ at  
 $\left(\theta_{\beta_*}, \alpha_{\beta_*} ,  l_{\beta_*}, z_{\beta_*} \right)$ which yields the operator 
  $$\mathscr{T}=(\mathscr{T}_1 , \mathscr{T}_2) :    [0,2\pi) \times \mathbb{R}^2 \times (\mathbb{X}_1)_{\mathbb{C}} \longrightarrow (\mathbb{Y} )_{\mathbb{C}} \times \mathbb{R},$$
 where
\begin{equation}\label{3.16}
\left\{\begin{aligned}
 \mathscr{T}_1\left(w_1, w_2, w_3, w_4 \right) = &  2\|\varphi_{*}\|_{Y_{\mathbb{C}}}^2 w_2
 \\
\mathscr{T}_2\left(w_1, w_2, w_3, w_4 \right) = & \left(d \Delta+ F(0,0)   - r(x)-\beta_* \right) w_4 - \frac{\mathrm{i} \varphi_{*}}{  \kappa (\beta_*)} w_3\\
&
+ \left(    
 \frac{\theta_4 \varphi_{*} }{\kappa(\beta_*)} \left[ F_{1}^{\prime}(0,0) - \beta_* q^{\prime}(0)   \right]   - \frac{\mathrm{i} l_{\beta_*}}{  \kappa (\beta_*)}    \right) 
 \varphi_{*} w_2 \\
&+ e^{- \mathrm{i} \theta_{\beta_*}}     \frac{\theta_4 \varphi_{*}   }{\kappa(\beta_*)}  F_{2}^{\prime}(0,0) 
\int_{\Omega} \mathcal{S}(\cdot, y) \varphi_* (y) \d y \, ( w_2 -\mathrm{i}  w_1),
\end{aligned}\right.
\end{equation}
and the perturbation $(w_1, w_2, w_3, w_4 )$ gives the direction and magnitude of a small change in each variable corresponding to $(\theta, \alpha,   l, z).$
If $\mathscr{T}$ is invertible, 
$H \left(\theta, \alpha, l, z, \beta \right)=0$
 is guaranteed, thereby completing the proof.

Firstly, we show that $\mathscr{T}$ is injective.  From $\mathscr{T}(w_1, w_2, w_3, w_4)=0$, it follows immediately from the first equation of \eqref{3.16} that $w_2 =0$. $\mathscr{T}_2=0$ is equivalent to 
 $$
\left(d \Delta+ F(0,0)   - r(x)-\beta_* \right) w_4 - \frac{\mathrm{i} \varphi_{*}}{  \kappa (\beta_*)} w_3
-
 \mathrm{i}  e^{- \mathrm{i} \theta_{\beta_*}}         
 \frac{\theta_4  \varphi_{*}  }{\kappa(\beta_*)}  F_{2}^{\prime}(0,0) 
\int_{\Omega} \mathcal{S}(\cdot, y) \varphi_* (y) \d y \, w_1=0,
 $$ where $\theta_4 $ in \eqref{c4}.
Multiplying both sides of the above equation by $\varphi_*(x)$ and integrating over $\Omega$ yields
$
 \mathrm{i}  w_3 + \mathrm{i}  e^{- \mathrm{i} \theta_{\beta_*}}        \theta_2 w_1=0.
 $
That is, 
\begin{equation}\label{3.166}
\begin{aligned}
\left\{\begin{array}{l}
Re:    
\theta_2 \sin(\theta_{\beta_*})
   w_1 =0  , \\
Im: 
  w_3 + \theta_2  \cos(\theta_{\beta_*})
  w_1     =0.
\end{array}\right.
\end{aligned}
\end{equation}

As we know that, under  $ (\mathrm{A}_2) $, we have $\theta_2 <0$ and $\sin(\theta_{\beta_*}) > 0,$ which implies that $w_1=0.$ Consequently, it follows from \eqref{3.166} and \eqref{3.16} that
$w_3=w_4=0.$  Hence $T$ is injective.  
 From Proposition 2.2 (iv), we know that $\mathcal{L}^{\beta_{*}} $ is a Fredholm operator with index zero. Therefore $\mathscr{T}$ is also a Fredholm operator of index zero, and since it is injective, it must be bijective.

By the implicit function theorem, there exist a constant $\delta_2$ and a continuously differentiable mapping $\beta \longmapsto \left(\theta_{\beta}, \alpha_{\beta},   l_{\beta}, z_{\beta} \right) \in     [0,2\pi) \times \mathbb{R}^2 \times (\mathbb{X}_1)_{\mathbb{C}}$ such that 
$$
H(\theta_{\beta}, \alpha_{\beta},   l_{\beta}, z_{\beta}, \beta )=0
\ \ \text{for} \ \beta \in \mathcal{N}_{-}(\beta_*, \delta_2 ).$$
Hence, the existence is established, and the uniqueness can be deduced by invoking the convergence argument introduced in Theorem \ref{th3.2}. It then follows that
$(\theta_{\beta}, \alpha_{\beta},   l_{\beta}, z_{\beta} ) \rightarrow \left(\theta_{\beta_*}, 1, l_{\beta_*}, z_{\beta_*} \right)$ as $\beta \rightarrow \beta_*,$ where $\theta_{\beta_*}$ and $l_{\beta_*}$ in \eqref{theta12}. Consequently,
for each $ \beta \in \mathcal{N}_{-}(\beta_*, \delta_2 ),$
 there exists a pair of purely imaginary eigenvalue $i \omega_{\beta} =  \frac{i l_\beta}{\kappa(\beta_*)}\left(\beta-\beta_* \right) $ corresponding to
   a sequence of critical delay values $\tau=\tau_{k}=\frac{\theta_{\beta} + 2k \pi}{\omega_{\beta}}, k \in \mathbb{N}_0.$ 
\end{proof}


In the following we analyze the transversality condition. 
By employing arguments similar to that used in   Lemma 4.4 of \cite{css2016} and Theorem 2.11 of \cite{css2012JDE}, we can show that 
$$
S^*_{k}(\phi^*_\beta, \phi_\beta) = \int_{\Omega} \overline{\phi_\beta^*}(x) \phi_\beta (x )\mathrm{d} x + \tau_k   e^{-\mathrm{i} \theta_\beta}  \int_{\Omega} \int_{\Omega} \mathcal{S}( x, y) B(x) u_\beta(x) \overline{\phi_\beta^*}(x) \phi_\beta (y) \mathrm{d} y \mathrm{~d} x \neq 0,
$$
and $\lambda= \mathrm{i} \omega_\beta $ is   a simple eigenvalue of $\mathscr{A}_{\tau_k, \beta}$ for $ k \in \mathbb{N}_0.$  Parallelly, $ -\mathrm{i} \omega_\beta$ is a simple eigenvalue of the adjoint
operator $\mathscr{A}^*_{\tau_k, \beta}$
 with adjoint eigenvector $\phi_\beta^*\in \mathbb{X}_{\mathbb{C}} \backslash\{0\}$ (has the similar form to $\phi_\beta$)
defined by
 $  \Pi^*( \tau_k, \beta, \overline{\mathrm{i} \omega_{\beta}} ) \phi^*_{\beta} = 0.$

 By differentiating $\Pi(\tau, \beta, \lambda(\tau)) \phi(\tau)=0$ with respect to \(\tau\) at $\tau=\tau_k$, one immediately finds 
$$
  \begin{aligned}
  &\frac{\d \lambda(\tau_k)}{\d \tau}   \left(   - \phi_\beta (x )  - \tau_k   e^{-\mathrm{i} \theta_\beta}    B(x) u_\beta(x) \int_{\Omega} \mathcal{S}( x, y)  \phi_\beta (y) \mathrm{d} y \right) \\
& + -\mathrm{i} \omega_\beta   e^{-\mathrm{i} \theta_\beta} B(x)  u_{\beta}(x)  \int_{\Omega} \mathcal{S}(x, y ) \phi_{\beta}(y) \d y  +\Pi(\tau_k, \beta , \mathrm{i} \omega_\beta    ) \frac{d \phi(\tau_k)}{\d \tau}=0.
\end{aligned}
$$
 Subsequently, multiplying the above equation by $\overline{\phi_\beta^*}(x)$ and integrating over the domain $\Omega$, we have  
$$
  \begin{aligned}
  \frac{\d \lambda(\tau_k)}{\d \tau}  =&  \frac{-\mathrm{i} \omega_\beta     e^{-\mathrm{i} \theta_\beta}}{S^*_{k}(\phi^*_\beta, \phi_\beta)}        \int_{\Omega} \int_{\Omega} B(x)  u_{\beta}(x) \mathcal{S}(x, y) \overline{\phi_\beta^*}(x) \phi_{\beta}(y) \d y \d x\\
  =&  \frac{-\mathrm{i} \omega_\beta     e^{-\mathrm{i} \theta_{\beta}}}{ |S^*_{k}(\phi^*_\beta, \phi_\beta)|^2}        \int_{\Omega} \int_{\Omega} B(x)  u_{\beta}(x) \mathcal{S}(x, y) \overline{\phi_\beta^*}(x) \phi_{\beta}(y) \d y \d x  \int_{\Omega} \overline{\phi_\beta^*}(x) \phi_\beta (x )\mathrm{d} x   \\
  & - \frac{\tau_k \mathrm{i} \omega_\beta  }{ |S^*_{k}(\phi^*_\beta, \phi_\beta)|^2}       \left| \int_{\Omega} \int_{\Omega} B(x)  u_{\beta}(x) \mathcal{S}(x, y) \overline{\phi_\beta^*}(x) \phi_{\beta}(y) \d y \d x \right|^2.   \\
\end{aligned}
$$
From  the expressions of $\phi_\beta(x), \theta_{\beta}$, $\omega_{\beta}$,  we can see that
$$
\phi_\beta(x) \rightarrow \varphi_*(x), \,   \phi^*_\beta(x) \rightarrow \varphi_*(x),   \,   \frac{\omega_\beta}{\beta-\beta_*}  \rightarrow \frac{l_{\beta_*}}{\kappa(\beta_*)} \text { as } \beta \rightarrow \beta_*.
$$
Then, we have
$$
\begin{aligned}
 \lim _{\beta \rightarrow \beta_*} \frac{1}{\left(\beta-\beta_*\right)^2} \frac{\d  \lambda \left(\tau_k\right) }{\d \tau} &=  \frac{-l_{\beta_*}  e^{-\mathrm{i} \theta_{\beta_*}}}{ \hat{S} \kappa(\beta_*)} \int_{\Omega} \int_{\Omega} F_{2}^{\prime}(0,0)  \frac{\theta_4 }{\kappa(\beta_*)} \mathcal{S}(x, y) \varphi_*^2(x) \varphi_*(y) \d y \d x  \int_{\Omega} \overline{\phi_\beta^*}(x) \phi_\beta (x )\mathrm{d} x \\
&  =\frac{-l_{\beta_*}  e^{-\mathrm{i} \theta_{\beta_*}} \theta_4^2 \theta_2}{\hat{S} \kappa^2(\beta_*)}, \, \text{ where } \hat{S}=\lim _{\beta \rightarrow \beta_*}\left|S^*_{k}(\phi^*_\beta, \phi_\beta)\right|^2 >0 \text{ is real. }
\end{aligned}
$$ 
 With condition $ (\mathrm{A}_2) $, we know that $\theta_2<0$ and $\sin \theta_{\beta_*} >0.$ Hence, taking the real part yields
$$ \lim _{\beta \rightarrow \beta_*} \frac{1}{\left(\beta-\beta_*\right)^2} \frac{\d Re( \lambda \left(\tau_k\right) )}{\d \tau}= \frac{-l_{\beta_*}  \sin\theta_{\beta_*} \theta_4^2 \theta_2}{\hat{S} \kappa^2(\beta_*)}>0.$$
Up to now, the transversality condition
     \begin{equation*}\label{retau}
          \frac{\d Re(\lambda(\tau_k))}{\d \tau}  > 0, \text{ for } k \in \mathbb{N}_0,
     \end{equation*} 
     is confirmed.
Consequently, we have the following main theorem.

\begin{theorem}\label{th4.8} 
Let $u_\beta(x)$ be the positive spatially  nonhomogeneous steady-state solution of   \eqref{1.1} formed by \eqref{ux}.    Assume that   the condition $ (\mathrm{A}_2) $ with $\beta \in \mathcal{N}_{-}(\beta_*, \delta_2)$ holds,   then   there exists a sequence $\left \{\tau_k\right\}_{k=0}^{\infty}$ such that a
 Hopf bifurcation occurs at $u_\beta(x)$ while $\tau=\tau_k, k \in \mathbb{N}_0$.
\end{theorem}

\begin{remark}\label{remark4.7}
 By similar analysis,
if $   \kappa(\beta_*)>0 \text { and } \tilde{\kappa}( \beta_* )<0, $ 
we can show the  Hopf bifurcation  occurs at $u_\beta(x)$ with $ \beta > \beta_*.$
\end{remark}
 
 The local stability and Hopf bifurcation of the positive steady state $u_\beta(x)$, as described in 
 Theorems  \ref{th3.2},  \ref{th4.8},  and Remarks \ref{remark3.3}, \ref{remark4.7}, 
 can be visualized in the  $(\theta_1, \theta_2)$ plane.

\begin{figure}[h] \centering \begin{minipage}{0.49\linewidth} \centerline{\includegraphics[width=7.5cm]{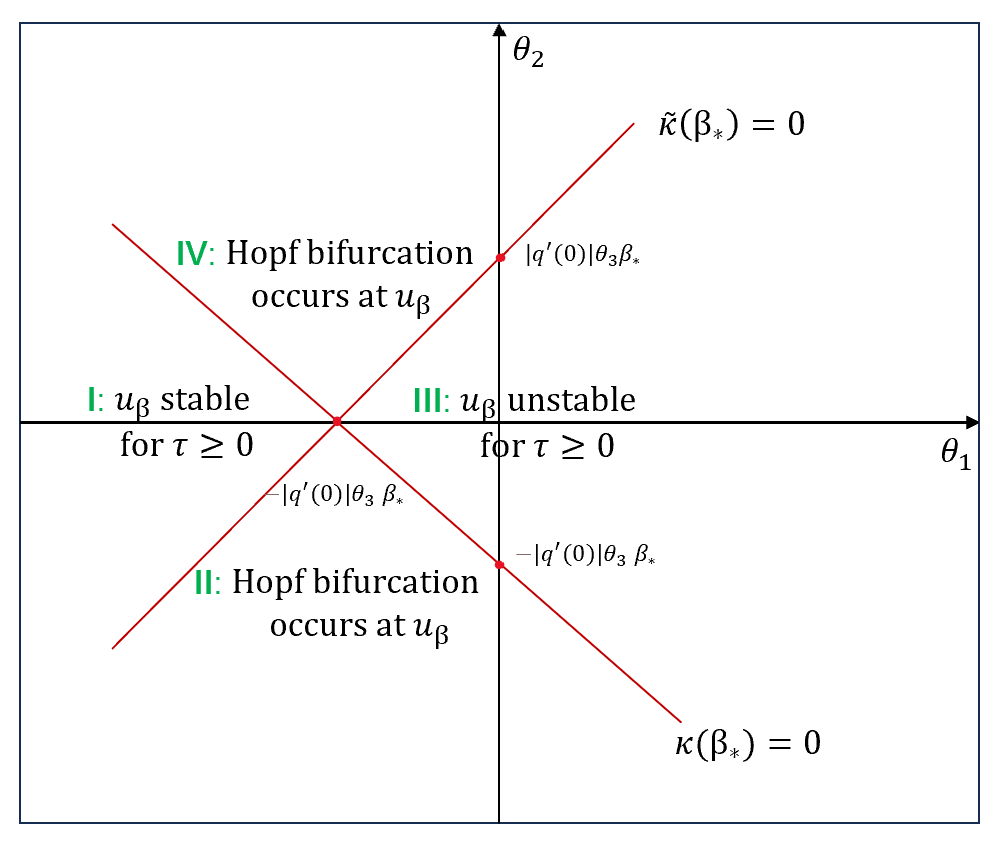}}	\end{minipage}	\captionsetup{margin=30pt}		\caption{ The local dynamics of $u_\beta(x)$. In region I, the bifurcated positive steady state $u_\beta$ is stable for $\tau \geq 0$. In region II and IV, Hopf bifurcation occurs at $u_\beta(x)$. In region III, $u_\beta(x)$ is always unstable for $\tau \geq 0$. }        \label{fig1} \end{figure}

\section{Example}

In this section, we present a numerical example to validate our theoretical findings  (Theorem \ref{th3.2}
and Theorem \ref{th4.8}) and to investigate the effects of treatment on the local asymptotic stability of positive steady-state solution, as well as the occurrence of Hopf bifurcation.

For the nonlinear functional form $F( \cdot , \cdot )$ in system \eqref{1.1}, we adopt the formulation introduced by Britton \cite{Britton1990}, namely
 $$
F( \cdot , \cdot )=1+a_1 u(x, t)-a_2 u^2(x, t)-(1+a_1-a_2) \int_{0}^{\pi} \mathcal{S}(x, y) u(y, t-\tau) \d y.
$$
To assess the impact of radiotherapy and chemotherapy on tumor cell dynamics, we perform numerical simulations to examine the asymptotic steady states of tumor cells across varying treatment intensity $\beta$.
In this setting, $\beta$ is assumed to increase monotonically with the treatment dose, implying that higher doses correspond to larger values of $\beta$.
Following the characterization in \cite{m2018}, 
the treatment-induced tumor cell death rate is modeled as
$$
 q(u)=(1-\frac{u}{u_{max}}),  \quad \beta= 1- e^{-\alpha_1 Dose- \alpha_2 Dose^2},
$$
where $e^{-\alpha_1 Dose- \alpha_2 Dose^2}$ denotes the surviving fraction, and $\alpha_1$ and $\alpha_2$ are radiosensitivity parameters. Here $Dose$ represents the delivered radiation dose. We impose 
Dirichlet boundary conditions on the special domain $\Omega=[0, \pi]$. For simplicity, we take
$$
   r(x)=r
 \text{ and }
 \mathcal{S}(x, y)= \sin x \sin y.
$$
Under these specifications, the system \eqref{1.1} can be rewritten as,
\begin{equation*}\label{sm5.1}
\left\{\begin{aligned}     &\frac{\partial u(x, t)}{\partial t} = d \Delta u(x,t)+   u(x, t) 
\left(
1+a_1 u(x, t)-a_2 u^2(x, t)-(1+a_1-a_2) \int_{0}^{\pi} \sin x \sin y  \, u(y, t-\tau) \d y
\right)\\
& \quad \quad \quad \quad \quad - (1- e^{-\alpha_1 Dose- \alpha_2 Dose^2})(1-\frac{u(x,t)}{u_{max}}) u(x,t) -r u(x,t),      \quad x \in (0, \pi), \,  t>0,\\     & u(0, t) =u(\pi,t)=0, \,  t>0.     \end{aligned}\right. \end{equation*} 
A straightforward calculation yields $\beta_*= -d+ 1 -r$, $\varphi_*(x)=\sin x$,  and 
$$
\theta_1= \frac{4}{3} a_1, \ \theta_2= \frac{2 \pi}{3} (a_2-a_1-1), \ \theta_3= \frac{4}{3}, \ \theta_4= \frac{\pi}{2}$$
in \eqref{c1c2} and \eqref{c4}, and $q^{\prime}(0)= -1/u_{max}$.

In the following simulation, we fix the parameter values as 
$$d=0.1, \ \alpha_1=0.2, \ \alpha_2=0.3,  \ u_{max}=1.$$

\subsection{Effect of treatment parameter $\beta$ on   the positive steady state $u_{\beta}(x)$}

We first consider the parameter set
$$ \ (I): \qquad
r=0.5, \ 
a_1=-0.49, \ a_2=0.5,  \ \tau=0.1.$$
In this case, the critical value of the treatment parameter is
$\beta_* = -d + 1-r=0.4 $.
As we know that, 
$$\kappa(\beta) = \theta_1 + \theta_2 - q'(0)\theta_3 \beta, \quad
\kappa(\beta_*) = \theta_1 - \theta_2 - q'(0)\theta_3 \beta_*.$$
According to Theorem \ref{ubx}, the positive steady state  $u_\beta(x)$ exists whenever $\kappa(\beta) \neq 0$. 
Fig.~\ref{kappavsbeta} (Right) illustrates the relationship between $\beta$ and $\kappa(\beta)$. 
Thus, for the chosen parameter set, it follows that $u_\beta(x)$ exists for 
$ 0< \beta < \beta_* $.

\begin{figure}[h]
	\centering
	\begin{minipage}{0.49\linewidth}
\centerline{\includegraphics[width=7.5cm]{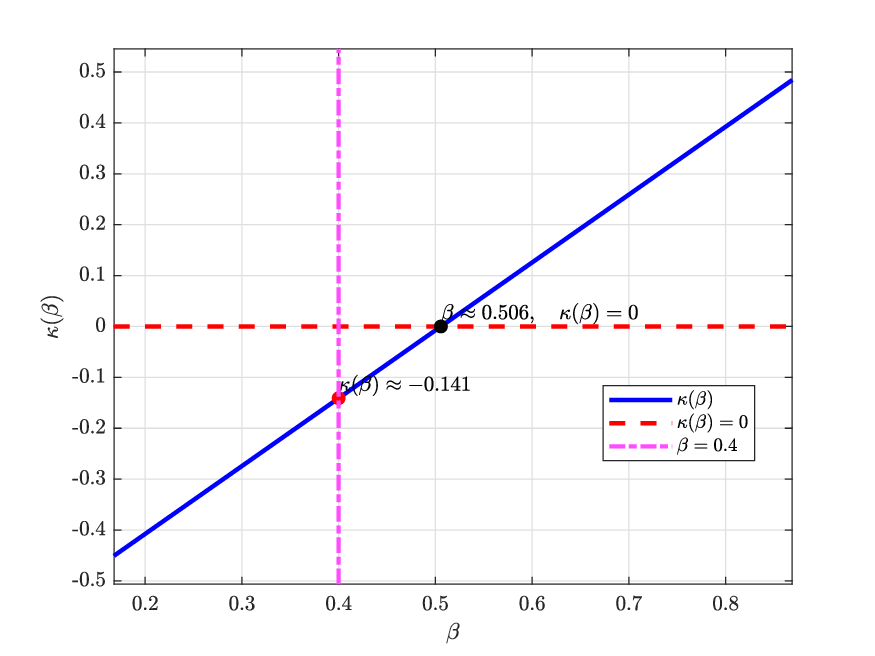}}
	\end{minipage}
	\hfill
	\begin{minipage}{0.49\linewidth}
\centerline{\includegraphics[width=7.5cm]{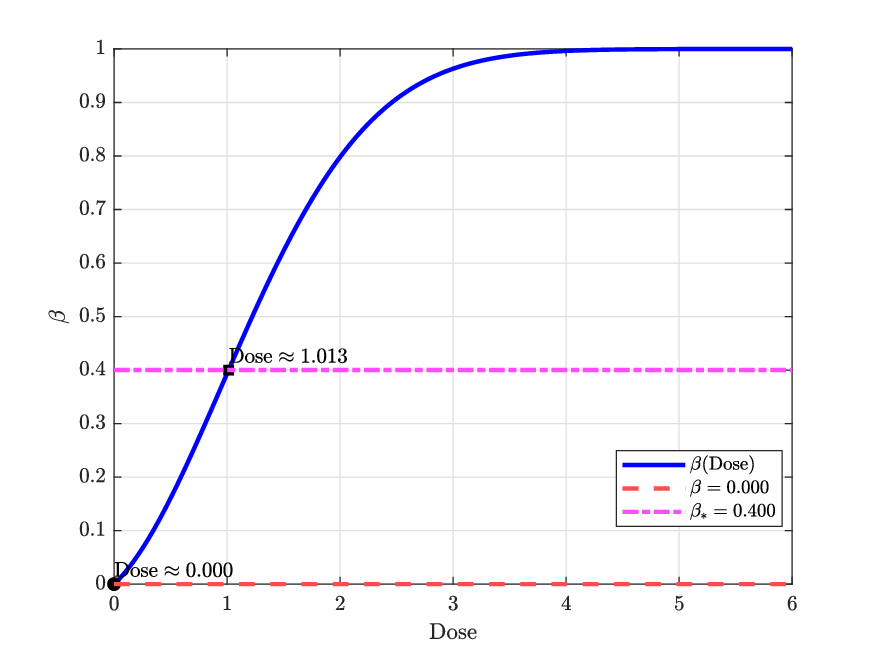}}
	\end{minipage}
	\hfill
	\captionsetup{margin=30pt}
		\caption{Left: $\kappa(\beta) $ vs. $\beta$ ; \qquad Right: Treatment rate $\beta$ vs. $Dose$.}
        \label{kappavsbeta}
\end{figure}
Since the treatment rate 
$\beta$ is assumed to increase monotonically with the clinical dose, following
$\beta = 1- e^{-\alpha_1 Dose- \alpha_2 Dose^2}$, we observe that when $\beta \in (0,0.4)$, the corresponding dose lies whthin the interval $(0,1.013)$, see Fig.~2 (Left).

When $ Dose=0.5$,  we have $\beta=0.161< \beta_*=0.4,$ and 
$$\kappa(\beta) = -0.46 \neq 0, \ 
 \kappa(\beta_*) = -0.141<0, \ \tilde{\kappa}(\beta_*) = -0.099<0.$$ 
These satisfy condition $(A_1)$ in Theorem \ref{th3.2}. 
Therefore, the theoretical analysis predicts that the spatially nonhomogeneous steady-state solution $u_{\beta}(x)$ is locally asymptotically stable.
This prediction is in excellent agreement with the numerical simulation results, as illustrated in Fig.~\ref{f2} (a).

If the dose is increased from $0.5$ to $1$, the corresponding 
treatment rate rises to $\beta=
0.393 < \beta_*=0.4.$ 
In this case, we have
$$\kappa(\beta) =
-0.15 \neq 0,
 \qquad \kappa(\beta_*) = -0.141<0, \qquad  \tilde{\kappa}(\beta_*) = -0.099<0.$$ 
Hence, condition  $ (\mathrm{A}_1) $ remains satisfied, ensuring the local stability of the spatially nonhomogeneous steady-state solution $u_{\beta}(x)$, as shown in Fig.~\ref{f2} (b).
By comparing (a) and (b) in Fig.~\ref{f2}, we observe that as the treatment rate $\beta$ (or equivalently, the dose) increases, the peak amplitude of the stabilized solution decreases.
This result suggests that, in clinical practice, tumor growth can be effectively suppressed by appropriately adjusting the therapy dose.

 \begin{figure}[h]
  \centering
  \captionsetup[sub]{labelformat=simple, labelfont=bf}
  \renewcommand\thesubfigure{(\alph{subfigure})}
  \begin{subfigure}[b]{0.32\textwidth}
    \centering
    \includegraphics[width=\linewidth]{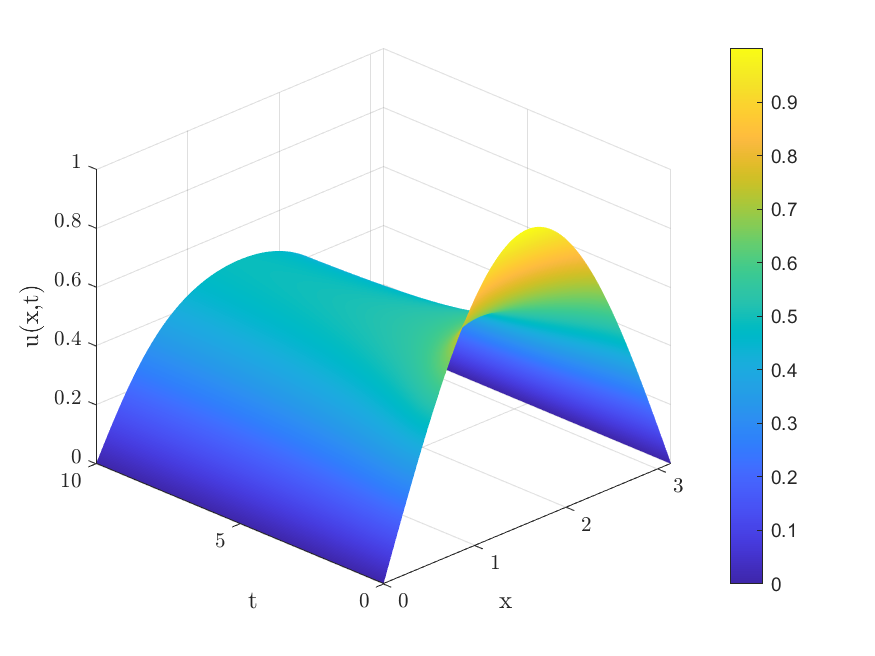}
    \caption{$\mathrm{Dose}=0.5,\ \tau=0.1$}
    \label{fig:dose05}
  \end{subfigure}
  \hfill
  \begin{subfigure}[b]{0.32\textwidth}
    \centering
    \includegraphics[width=\linewidth]{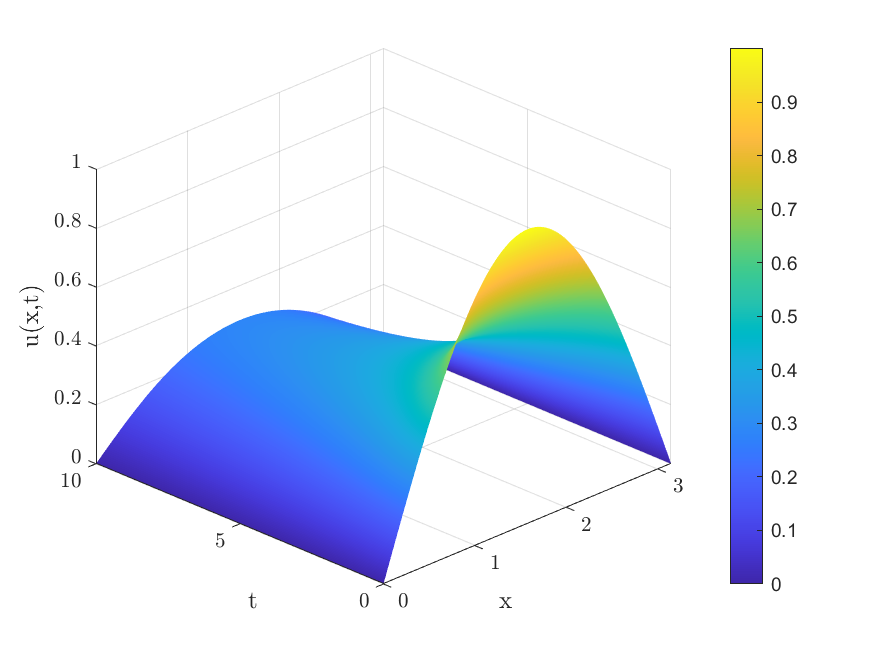}
    \caption{$\mathrm{Dose}=1,\ \tau=0.1$}
    \label{fig:dose1_tau01}
  \end{subfigure}
  \hfill
  \begin{subfigure}[b]{0.32\textwidth}
    \centering
    \includegraphics[width=\linewidth]{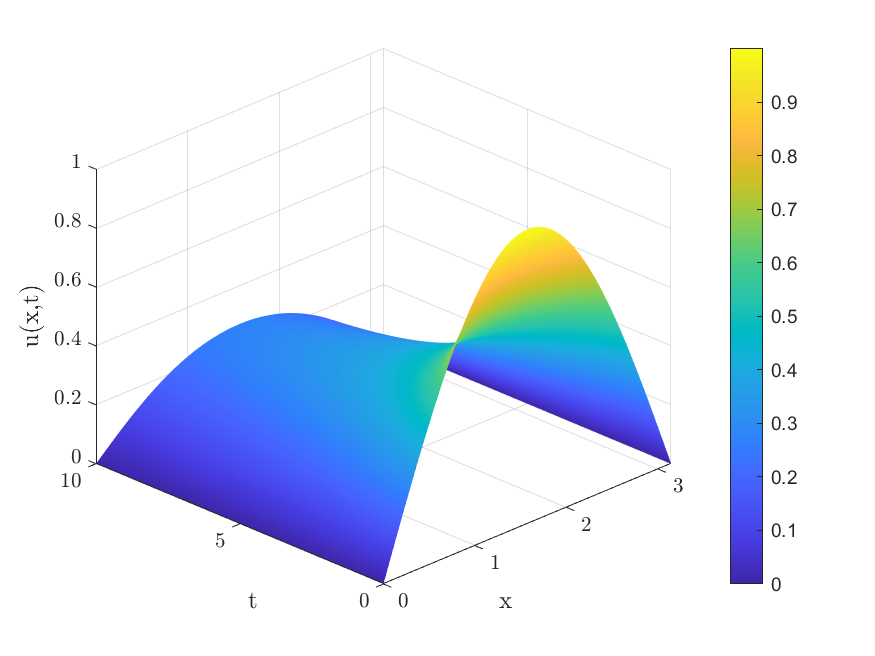}
    \caption{$\mathrm{Dose}=1,\ \tau=10$}
    \label{fig:dose1_tau10}
  \end{subfigure}
  \captionsetup{margin=30pt}
  \caption{Stability of the positive spatially nonhomogeneous steady state $u_{\beta}(x)$ under condition $(\mathrm{A}_1)$ with varied values of $Dose$ and delay. 
  }
  \label{f2}
\end{figure}

Even when the delay $\tau$ is increased from 0.1 to a large value, for example $\tau=10$, while keeping all other parameters the same as those in Fig.~\ref{f2} (b), the positive steady-state solution $u_{\beta}(x)$ remains locally asymptotically stable, as shown in Fig.~\ref{f2} (c). 
This observation confirms that,
under the conditions specified in Theorem \ref{th3.2}, such stability persists for all $\tau \geq 0.$

\subsection{Effect of treatment parameter $\beta$ on   Hopf Bifurcation}

In this subsection, we investigate the occurrence of Hopf bifurcation near the positive steady-state solution $u_{\beta}(x)$ with respect to the treatment rate $\beta$. 
Similar to the previous subsection, 
we first determine the admissible ranges of the treatment parameter 
$\beta$ and the $Dose$. 
We take the parameters set 
$$ (II): \qquad   r=0.1, \ 
 a_1=2, \ a_2=0.9, \  \tau=1.75.$$
For this choice of parameters, we obtain the critical value $\beta_*=0.8,$
and it follows that $\kappa(\beta)<0,$ for all $\beta< \beta_*.$
Accordingly, 
the admissible range of the treatment parameter
$\beta$ is $\beta \in (0, 0.8=\beta_*),$ while the corresponding
admissible range of the $Dose$ is approximately $(0,2.007)$.
 
When $Dose=0.4,$
under the parameter set (II),
we have $ \beta=0.12 < \beta_*=0.8$ and 
 $$\kappa(\beta)=  -1.571 \neq 0 , \  \kappa(\beta_*)=  -0.665<0,  \ \tilde{\kappa}(\beta_*)=8.132>0, $$ 
which satisfy condition $(\mathrm{A}_2)$ in Theorem \ref{th4.8}. Under these conditions, a periodic solution of $u(x,t)$ emerges via Hopf bifurcation at 
 $u_{\beta}(x)$, as illustrated in Fig. \ref{fig42}.
  To clearly display the oscillatory behavior, the right panel presents a 
cross-sectional view of $u(x,t)$ over the interval $t \in [0,100]$.

\begin{figure}[h]
	\centering
	\begin{minipage}{0.49\linewidth}
\centerline{\includegraphics[width=7.5cm]{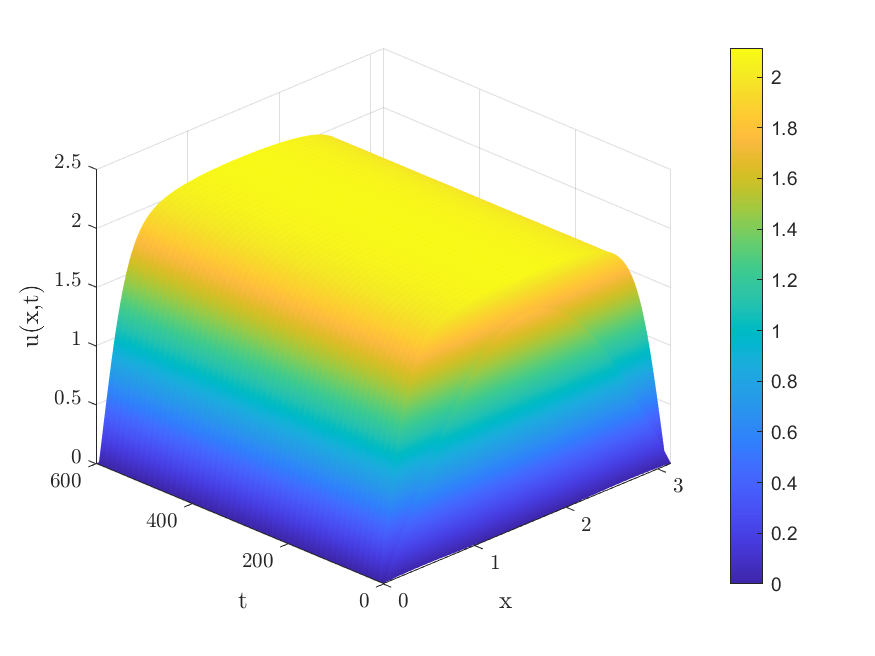}}
	\end{minipage}
	\hfill
	\begin{minipage}{0.49\linewidth}
\centerline{\includegraphics[width=7.5cm]{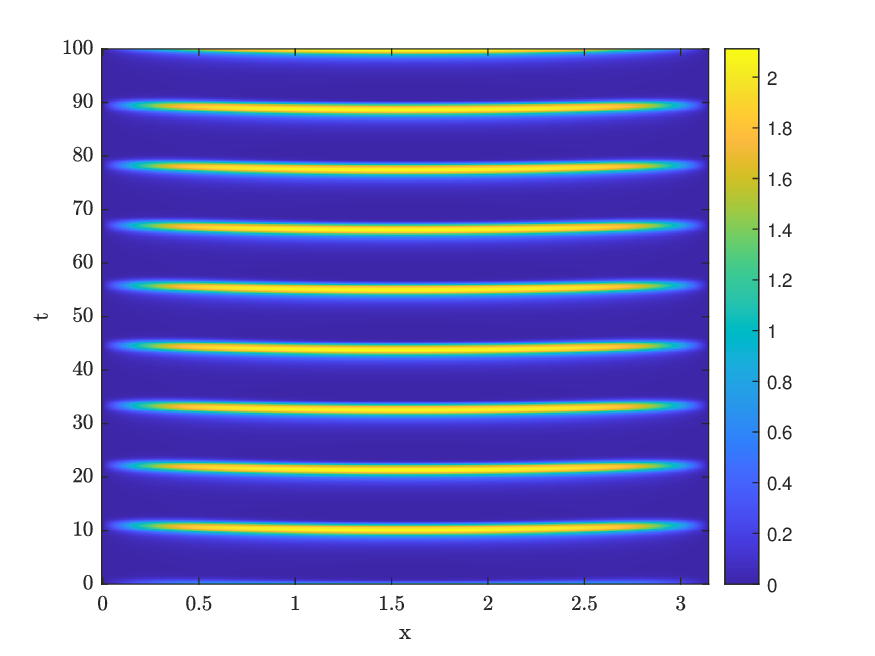}}
	\end{minipage}
	\hfill
	\captionsetup{margin=30pt}
		\caption{ 
        Occurrence of Hopf bifurcation when $ (\mathrm{A}_2) $ holds.
Left: Bifurcated periodic solution; Right: The cross-sectional
view of the solution.}
    \label{fig42}
\end{figure}

Next
 we increase the treatment parameter $\beta$  by raising  the delivered radiation therapy $Dose$ from $0.4$ to $1.2$, while keeping all other parameters unchanged. In this case, we have 
$$\beta_*=0.8,  
 \ \beta= 0.489 < \beta_*, \ 
\kappa(\beta) =  -1.079 \neq 0, \ 
 \kappa(\beta_*) = 
-0.665<0, \ \tilde{\kappa}(\beta_*)  = 8.132>0.$$  
These values satisfy condition  $ (\mathrm{A}_2) $ in Theorem \ref{th4.8}, implying that a Hopf bifurcation occurs at  $u_\beta(x)$. The resulting periodic solution is shown in Fig. \ref{fig5}.

\begin{figure}[h]
	\centering
	\begin{minipage}{0.49\linewidth}
\centerline{\includegraphics[width=7.5cm]{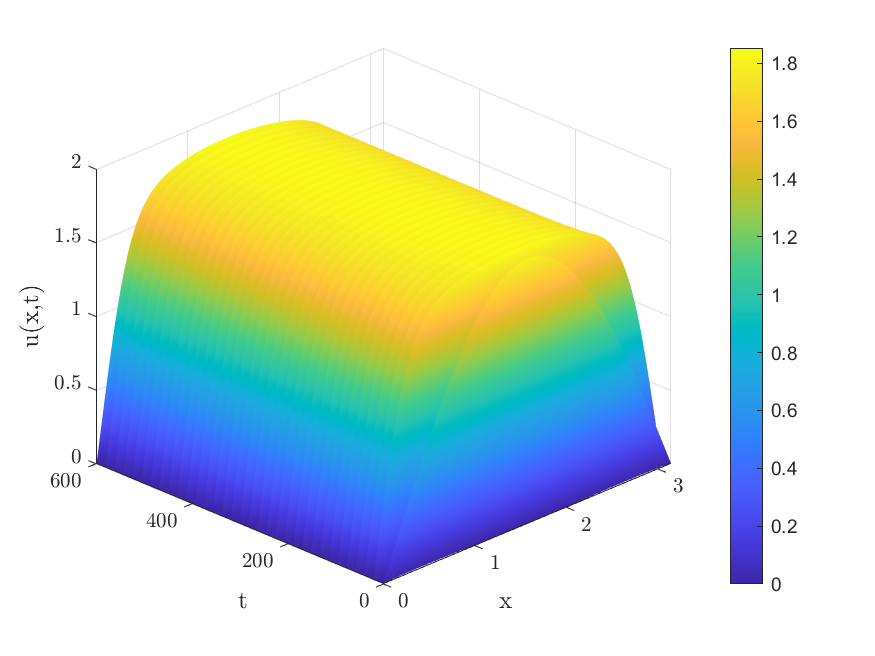}}
	\end{minipage}
	\hfill
	\begin{minipage}{0.49\linewidth}
\centerline{\includegraphics[width=7.5cm]{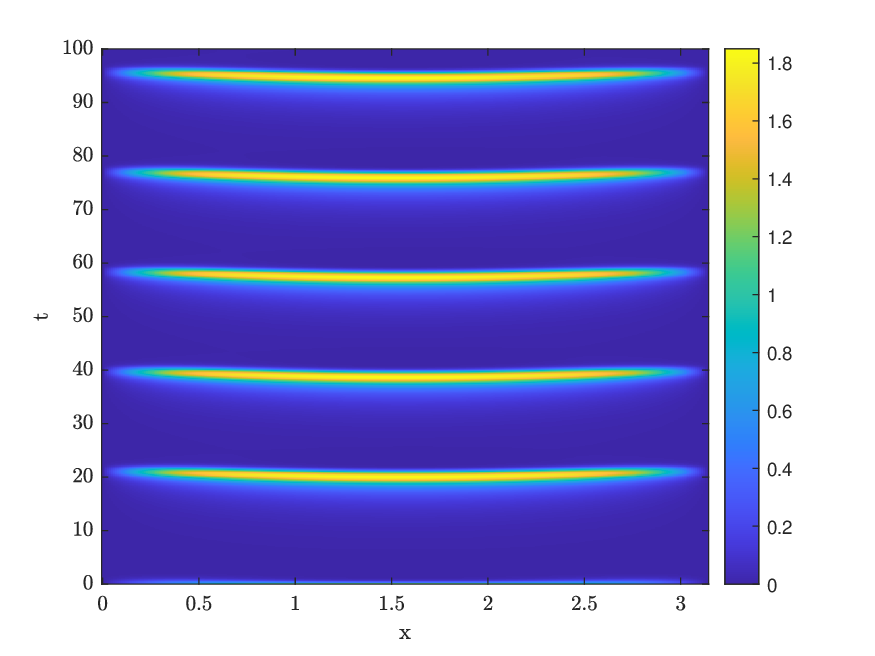}}
	\end{minipage}
	\hfill
	\captionsetup{margin=30pt}
		\caption{Occurrence of Hopf bifurcation when condition $(\mathrm{A}_2)$ is satisfied.
Left: Periodic solution arising from the Hopf bifurcation.
Right: Cross-sectional view of the solution with respect to $x$ and $t$
}
\label{fig5}
\end{figure}

By comparing the cross-sectional views on the right panels of Fig.~\ref{fig42} and Fig.~\ref{fig5}, one can clearly observe that as $Dose$ increases from $0.4$ to $1.2$, the spacing between the density peaks of tumor cells expands from approximately 10 to 20, while the amplitude of the periodic solution decreases from above 2 to below 2.
This reduction in peak amplitude, accompanied by an increase in oscillation period, suggests that under the present conditions, a moderate increase in drug dosage may lead to more stable and controlled tumor growth dynamics, thereby potentially improving the patient's overall quality of life.

\section{Conclusion}

In this paper, we have investigated a nonlocal, delayed reaction-diffusion tumor model that incorporates therapeutic effects under Dirichlet and Neumann boundary conditions.
By applying the Lyapunov-Schmidt reduction method, we established the existence of nontrivial steady-state solutions bifurcating from the trivial steady state. In the special case where $\kappa(\beta) \neq 0$, we further derived an approximate expression for the spatially nonhomogeneous positive steady state $u_{\beta}(x)$, providing analytical insight into the spatial structure of tumor cell distributions.

We then examined the stability of the positive steady state and explored the emergence of Hopf bifurcation within the system \eqref{1.1}.
Specifically, under condition $(\mathrm{A}_1)$ in Theorem~\ref{th3.2}, when the treatment parameter $ \beta \in \mathcal{N}_{-}(\beta_*, \delta_1 )$, the positive steady state $u_{\beta}(x)$ is locally asymptotically stable for all $\tau \geq 0$.
In contrast, under condition $(\mathrm{A}_2)$ in Theorem~\ref{th4.8}, when $\beta \in \mathcal{N}_{-}(\beta_*, \delta_2)$, the system undergoes a Hopf bifurcation at $u_\beta(x)$ as $\tau$ passes through critical values $\tau_k$, $k \in \mathbb{N}_0$, giving rise to periodic oscillations in tumor cell density.

Through illustrative numerical examples, we analyzed how the treatment parameter $\beta$ and the delivered radiation therapy dose jointly influence the steady-state and bifurcation behavior. Under the hypotheses of Theorem~\ref{th3.2}, increasing the $Dose$ leads to a decrease in the steady-state tumor cell density, indicating enhanced treatment efficacy (see Fig.~\ref{f2}). Furthermore, the comparative analysis of cross-sectional profiles revealed that moderate increases in $\beta$ (or equivalently, treatment intensity) tend to reduce the amplitude of oscillatory tumor dynamics while extending their temporal period (see Fig.~\ref{fig42} and Fig.~\ref{fig5}). This suggests that controlled adjustments in therapeutic dosage can stabilize tumor growth patterns, potentially delaying recurrence and improving treatment outcomes.

The theoretical framework developed here provides a foundation for integrating spatial and temporal factors into treatment planning. It can inform future studies on personalized, adaptive therapy strategies that exploit nonlinear dynamical features of tumor evolution. The techniques are sufficiently general
to apply to other nonlocal delayed reaction-diffusion systems beyond the biological
example.
Future research may extend this model by incorporating heterogeneous tissue structures, immune response mechanisms, or multi-drug interactions. Such extensions would further bridge the gap between mathematical modeling and clinical application, enhancing our understanding of spatiotemporal tumor dynamics under complex therapeutic interventions.

\section*{Acknowledgements}
The authors gratefully acknowledge professors Shanshan Chen and  Jie Liu for their invaluable guidance during the preparation of this manuscript.

\section*{Funding}  DH is funded by the China Scholarship Council (CSC) 202306410009. YY and DH are funded by the Natural Sciences and Engineering Research Council of Canada (NSERC) 257109.

\phantomsection

\end{document}